\documentclass[12pt,a4paper,notitlepage]{article}
\usepackage{amsthm}
\usepackage{amsfonts}
\usepackage{mathtools}

\usepackage{graphicx}

\usepackage{pdfpages}

\usepackage{float}

\usepackage{cite}
\usepackage{indentfirst}
\usepackage{color}
\usepackage{hyperref}
\hypersetup{
    colorlinks,
    citecolor=black,
    filecolor=balck,
    linkcolor=red,
    urlcolor=black
}

\usepackage{blindtext}
\usepackage[utf8]{inputenc}

\theoremstyle{plain}
\newtheorem{thm}{Theorem}[section] 

\newtheorem{prop}[thm]{Proposition}
\newtheorem{cor}[thm]{Corollary}
\newtheorem{lem}[thm]{Lemma}
\newtheorem{claim}{Claim}
\theoremstyle{definition}
\newtheorem{defn}[thm]{Definition}
\newtheorem{PC}{\textit{Proof of Claim}}

\newtheorem{rmk}{Remark}

\usepackage{amssymb}

\DeclareMathOperator*{\order}{list}

\newcommand{\RN}[1]{%
  \textup{\uppercase\expandafter{\romannumeral#1}}%
}

\newcommand{\Addresses}{{
  \bigskip
  \footnotesize

 \textsc{Department of Mathematical Sciences, Durham University,
    School Road, Durham, UK,  DH1 3LE}\par\nopagebreak
  \textit{E-mail address}: \texttt{j.m.wilson2@durham.ac.uk} }}

\usepackage[justification=centering]{caption}

\usepackage{caption}
\usepackage{subcaption}
\usepackage{rotating}

\title{Shellability and Sphericity of the quasi-arc complex of the M\"obius strip.}
\author{Jon Wilson}
\date{}

\usepackage{fancyhdr}

\fancypagestyle{plain}{
  \fancyhf{}
  \fancyfoot[C]{\thepage}%
}
\fancypagestyle{firstpage}{
  \fancyhf{}
  \fancyfoot[L]{\footnotesize{The author's PhD studies are funded by an EPSRC studentship.}}%
}

\begin{document}

\maketitle

\thispagestyle{firstpage}

\begin{abstract}

Shellability of a simplicial complex has many useful structural implications. In particular, it was shown in \cite{shelling} by Danaraj and Klee that every shellable pseudo-manifold is a PL-sphere. The purpose of this paper is to prove the shellability of the quasi-arc complex of the M\"obius strip. Along the way we provide elementary proofs of the shellability of the arc complex of the $n$-gon and the cylinder. In turn, applying the result of Danaraj and Klee, we obtain the sphericity of all of these complexes.
\end{abstract}

\pagenumbering{arabic}

\tableofcontents

\section{Introduction}

The arc complex $Arc(S)$ of a marked orientable surface $S$ was introduced and studied by Harer \cite{harer} whilst investigating the homology of mapping class groups of orientable surfaces. In \cite{fst},\cite{length}, Fomin, Shapiro and Thurston found there is a strong relation between cluster algebras and these arc complexes. They showed that $Arc(S)$ is a subcomplex of the cluster complex $\Delta(S)$ associated to the cluster algebra arising from $S$. Moreover, it was shown by Fomin and Zelevinsky, almost at the birth of cluster algebras, that the cluster complex of a cluster algebra has a polytopal realisation when the complex is finite, see \cite{fz-finite}. Since $Arc(S)$ and $\Delta(S)$ coincide when $S$ is an unpunctured surface, as a specific case, the well known fact that $Arc(n\textendash gon)$ is polytopal was rediscovered. Namely, it is dual to the associahedron.

In \cite{quasi} Dupont and Palesi consider the quasi-arc complex of unpunctured non-orientable surfaces. Imitating the approach in \cite{length} they describe how the `lengths' of quasi-arcs are related. In doing so they discover what the analogue of a cluster algebra arising from non-orientable surfaces should be. A natural question is to ask what kind of structure the quasi-arc complex has in this setting. Here, in some sense, the marked M\"{o}bius strip $\textnormal{M}_n$ plays the role of the $n\textendash gon$ - being the only non-orientable surface yielding a finite quasi-arc complex. 

For $n \in \{1,2,3\}$ it is easy to check that the quasi-arc complex $Arc(\textnormal{M}_n)$ of the M\"obius strip is a polytope, see Figure \ref{fig:polytopes}. However, in general, due to the absence of a root system it is difficult to find a polytopal realisation.

It is shown in \cite{mani} that shellability of a simplicial complex is a necessary condition for it being polytopal. This paper is concerned with proving the following theorem. \newline

\noindent \textbf{Main Theorem.} (Theorem \ref{Main}). \textit{$Arc(\textnormal{M}_n)$ is shellable for $n \geq 1$}. \newline

Shellability has its roots in polytopal theory where it turned out to be the missing piece of the puzzle for obtaining the Euler-Poincare formula. It has subsequently become a well established idea in combinatorial topology and geometry having some useful implications. For instance, Danaraj and Klee showed in \cite{shelling} that every shellable pseudo-manifold is a PL-sphere. As a consequence, we obtain the following result. \newline

\noindent \textbf{Corollary.} (Corollary \ref{sphere}). \textit{$Arc(\textnormal{M}_n)$ is a PL $(n-1)$-sphere for $n \geq 1$}. \newline

The paper is organised as follows. In Section \ref{quasi-cluster} we recall the work of Dupont and Palesi in \cite{quasi}. Here we define the quasi-arc complex of a non-orientable surface and discuss why it is a pseudo-manifold, and when it is finite.

In Section \ref{shellability} we firstly define shellability and recall some fundamental results. Next we restrict our attention to the $n\textendash gon$ and to $C_{n,0}$ - the cylinder with $n$ marked points on one boundary component, and none on the other. In the interest of introducing key ideas of the paper early on, we present a short proof that both $Arc(n\textendash gon)$ and $Arc(C_{n,0})$ are shellable. As a consequence, applying the result of Danaraj and Klee, we rediscover the classical fact of Harer \cite{harer2} that $Arc(n\textendash gon)$ and $Arc(C_{n,0})$ are PL-spheres.

Section \ref{main section} is dedicated to proving the shellability of $Arc(\textnormal{M}_n)$ and occupies the bulk of the paper.

\section*{\large \centering Acknowledgements}
I would like to thank Carsten Lange for proposing shellability as an interesting property to study, and additionally thank him and Frederic Palesi for helpful discussions. I especially thank my supervisor Anna Felikson for her continued support and encouragement whilst writing this paper. Finally, I thank Michael Wilson for suggesting a neat application for drawing pictures.

\section{Quasi-cluster algebras}
\label{quasi-cluster}

This section recalls the work of Dupont and Palesi in \cite{quasi}. \newline

Let $S$ be a compact $2$-dimensional manifold with boundary $\partial S$. Fix a set $M$ of marked points in $\partial S$. The tuple $(S,M)$ is called a bordered surface. We wish to exclude the cases where $(S,M)$ does not admit any triangulation. As such, we do not allow $(S,M)$ to be a monogon, digon or triangle.

\begin{defn}

An \textit{\textbf{arc}} is a simple curve in $(S,M)$ connecting two (not necessarily distinct) marked points.

\end{defn}

\begin{defn}

A closed curve in $S$ is said to be \textit{\textbf{two-sided}} if it admits a regular neighbourhood which is orientable. Otherwise, it is said to be \textit{\textbf{one-sided}}.

\end{defn}

\begin{defn}

A \textit{\textbf{quasi-arc}} is either an arc or a simple one-sided closed curve in the interior of $S$. Let $A^{\otimes}(S,M)$ denote the set of quasi-arcs in $(S,M)$ considered up to isotopy.

\end{defn}

It is well known that a closed non-orientable surface is homeomorphic to the connected sum of $k$ projective planes $\mathbb{R}P^2$. Such a surface is said to have (non-orientable) genus $k$. Recall that the projective plane is homeomorphic to a hemisphere where antipodal points on the boundary are identified. A \textit{\textbf{cross-cap}} is a cylinder where antipodal points on one of the boundary components are identified. We represent a cross cap as shown in Figure \ref{fig:crosscap}.

Hence, a closed non-orientable surface of genus $k$ is homeomorphic to a sphere where $k$ open disks are removed, and have been replaced with cross-caps. More generally, a compact non-orientable surface of genus $k$, with boundary, is homeomorphic to a sphere where more than $k$ open disks are removed, and $k$ of those open disks have been replaced with cross-caps.

\begin{figure}[H]
\centering
\includegraphics[width=30mm]{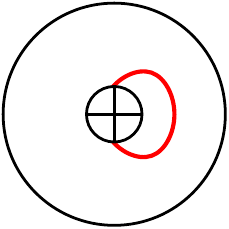}
\caption{A picture of a crosscap together with a one-sided closed curve.}
\label{fig:crosscap}
\end{figure}

\begin{defn}

Two elements in $A^{\otimes}(S,M)$ are called \textit{\textbf{compatible}} if there exists representatives in their respective isotopy classes that do not intersect in the interior of $S$.

\end{defn}

\begin{defn}

A \textit{\textbf{quasi-triangulation}} of $(S,M)$ is a maximal collection of mutually compatible arcs in $A^{\otimes}(S,M)$. A quasi-triangulation is called a \textit{\textbf{triangulation}} if it consists only of arcs, i.e there are no one-sided closed curves.

\end{defn}

\begin{prop}[\cite{quasi}, Prop. 2.4.]
\label{flip}
Let $T$ be a quasi-triangulation of $(S,M)$. Then for any $\gamma \in T$ there exists a unique $\gamma' \in A^{\otimes}(S,M)$ such that $\gamma \neq \gamma'$ and $\mu_{\gamma}(T) := T\setminus \{\gamma\} \cup \{\gamma'\}$ is a quasi-triangulation of $(S,M)$.

\end{prop}

\begin{defn}
$\mu_{\gamma}(T)$ is called the \textit{\textbf{quasi-mutation}} of $T$ in the direction $\gamma$, and $\gamma'$ is called the \textit{\textbf{flip}} of $\gamma$ with respect to $T$.

\end{defn}

The flip graph of a bordered surface $(S,M)$ is the graph with vertices corresponding to (quasi) triangulations and edges corresponding to flips.  It is well known that the flip graph of triangulations of $(S,M)$ is connected. Moreover, it can be seen that every one-sided closed curve, in a quasi-triangulation $T$, is bounded by an arc enclosing a M\"obius strip with one marked point on the boundary. Therefore, if we perform a quasi-flip at each one-sided closed curve in $T$ we arrive at a triangulation. As such, we get the following proposition.

\begin{prop}[\cite{quasi}, Prop. 2.12.]
\label{flip connected}
The flip graph of quasi-triangulations of $(S,M)$ is connected.

\end{prop}

\begin{cor}
\label{rank}
The number of quasi-arcs in a triangulation of $(S,M)$ is an invariant of $(S,M)$.

\end{cor}

\begin{defn}

The \textit{\textbf{quasi-arc complex}} $Arc(S,M)$ is the simplicial complex on the ground set $A^{\otimes}(S,M)$ such that $k$-simplices correspond to sets of $k$ mutually compatible quasi-arcs. In particular, the vertices in $Arc(S,M)$ are the elements of $A^{\otimes}(S,M)$ and the maximum simplices are the quasi-triangulations.

\end{defn}

Together, Corollary \ref{rank} and Proposition \ref{flip} prove the following proposition.

\begin{prop}

$A^{\otimes}(S,M)$ is a pseudo-manifold. That is, each maximal simplex in $A^{\otimes}(S,M)$ has the same cardinality, and each simplex of co-dimension $1$ is contained in precisely two maximal simplices.

\end{prop}

\begin{thm}[\cite{quasi}, Theorem 7.2]

Given a non-orientable bordered surface $(S,M)$ then $Arc(S,M)$ is finite if and only if $(S,M)$ is $\textnormal{M}_n$, the M\"obius strip with $n$ marked points on the boundary.

\end{thm} 

Moreover, $Arc(\textnormal{M}_n$) has some seemingly nice properties. Figure \ref{fig:polytopes} shows that for $n \in \{1,2,3\}$ it is polytopal.

\begin{figure}[H]
\centering
\includegraphics[width=133mm]{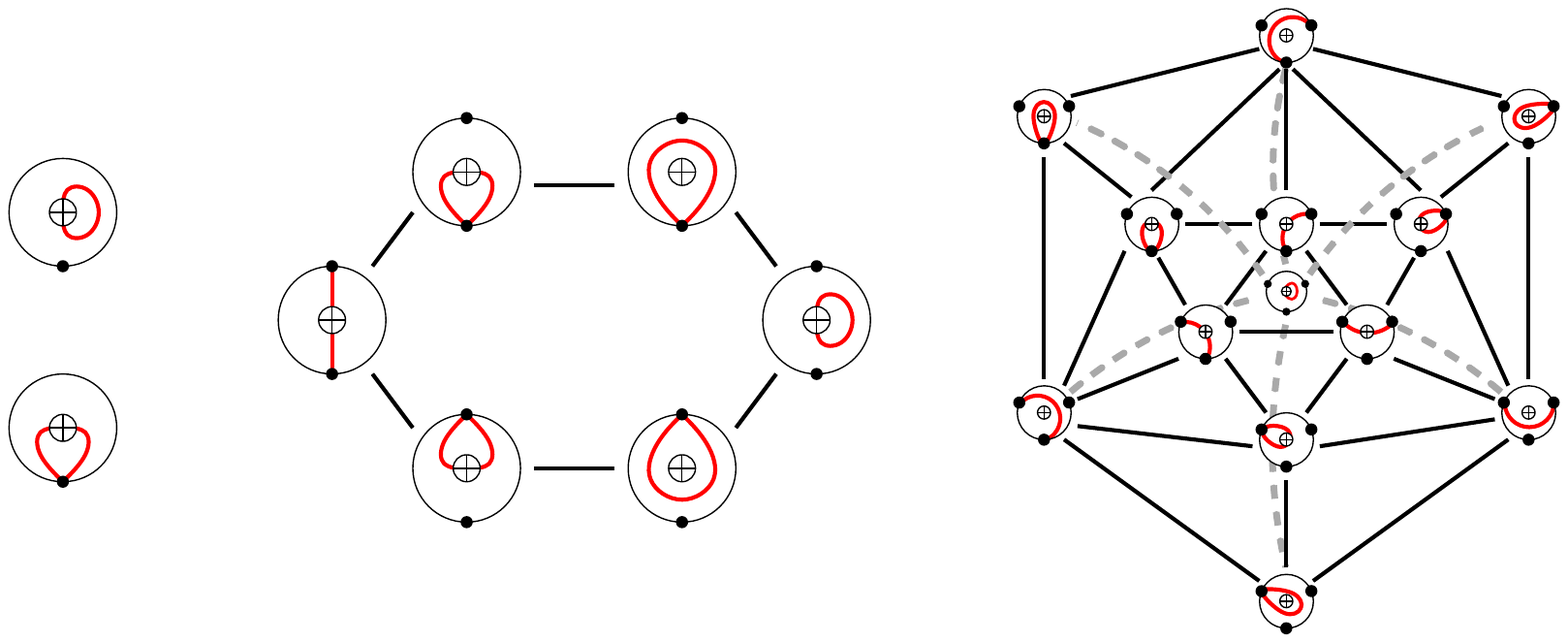}
\caption{The quasi-arc complexes of $\text{M}_1,\text{M}_2$ and $\text{M}_3$.}
\label{fig:polytopes}
\end{figure}

\section{Shellability}
\label{shellability}

In this section we recall some basic facts about shellability, and introduce the fundamental ideas used throughout this paper.

\subsection{Definition of shellability and basic facts.}

\begin{defn}  An $n$-dimensional simplicial complex is called \textit{\textbf{pure}} if its maximal simplices all have dimension $n$.

\end{defn}

\begin{defn}

Let $\Delta$ be a finite (or countably infinite) simplicial complex. An ordering $C_1, C_2 \ldots $ of the maximal simplices of $\Delta$ is a \textit{\textbf{shelling}} if the complex $B_k := \big( \bigcup_{i=1}^{k-1}C_i\big) \cap C_k$ is pure and $(dim(C_k) -1)$-dimensional for all $k \geq 2$.

\end{defn}

\begin{defn}

The \textit{\textbf{simplicial join}} $\Delta_1 \ast \Delta_2$ of two simplicial complexes $\Delta_1$ and $\Delta_2$ on disjoint ground sets has its faces being sets of the form $\sigma_1 \cup \sigma_2$ where $\sigma_1 \in \Delta_1$ and $\sigma_2 \in \Delta_2$.

\end{defn}

The following proposition is a simple and well known result. For instance, see \cite{bjorner}.

\begin{prop}
\label{join}
The simplicial join $\Delta_1 \ast \Delta_2$ is shellable \textit{if and only if} the simplicial complexes $\Delta_1 , \Delta_2$ are both shellable.

\end{prop}

In particular, Proposition \ref{join} tells us that the cone over a shellable complex is itself shellable.

\begin{prop}
\label{Shelling}

If $\Delta = Arc(S,M)$ then finding a shelling for $\Delta$ is equivalent to ordering the set of triangulations $T_i$ of $(S,M)$ so that $\forall k$ and $\forall j < k$ $\exists i < k $ such that $T_i$ is related to $T_k$ by a mutation and $T_j \cap T_k \subseteq T_i \cap T_k$.

\begin{proof} 

Note that triangulations $T_i$ of $S$ correspond to maximal simplices in $Arc(S,M)$ and that partial triangulations $T_i \cap T_j$ correspond to simplices of $Arc(S)$. Note that $T_i \cap T_k$ is a $(dim(T_k) -1)$-simplex $\textit{iff}$ $T_k$ is a mutation away from $T_k$. Furthermore, since $B_k := \big( \bigcup_{i=1}^{k-1}T_i\big) \cap T_k$ must be pure and $(dim(T_k) -1)$-dimensional for all $k \geq 2$, it follows that $B_k$ is the union of $(dim(T_k) -1)$-simplices. So we must have that $\forall j < k$ $\exists i < k $ such that $T_i$ is a mutation away from $T_k$ and the partial triangulation $T_j \cap T_k$ is a face of $T_i \cap T_k$ (i.e $T_j \cap T_k \subseteq T_i \cap T_k$).

\end{proof}

\end{prop}

Proposition \ref{Shelling} motivates Definition \ref{shelling defn}.

\begin{defn}
\label{shelling defn}
Given a subcollection of triangulations $\Gamma$ of a surface $S$ call $\Gamma$ \textit{\textbf{shellable}} if it admits an ordering of $\Gamma$ such that $\forall k$ and $\forall j < k$ $\exists i < k $ such that $T_i$ is related to $T_k$ by a mutation and $T_j \cap T_k \subseteq T_i \cap T_k$.

\end{defn}

\begin{defn}

We say two sets of triangulations $A$, $B$ are \textit{\textbf{equivalent}} if their induced simplicial complexes are isomorphic, up to taking cones. If $A$ and $B$ are equivalent we write $A \equiv B$.

\end{defn}

\begin{rmk}

Let $\Delta_{A}$ denote the induced simplicial complex of a set of triangulations $A$. Note that taking a cone over $\Delta_{A}$ can be thought of as disjointly adding one particular arc to every triangulation in $A$.

\end{rmk}

The following proposition is just a special case of Proposition \ref{join}.

\begin{prop}

If $A \equiv B$ then $A$ is shellable \text{if and only if} $B$ is shellable.

\end{prop}

\noindent \textbf{\underline{Notation}:}

\begin{itemize}

\item $\displaystyle\order_{i=1}^{n} x_i$ is the ordering $x_1,x_2,\ldots , x_n$ of the set $\{x_i | 1\leq i \leq n\}$. \\

\item $\displaystyle\order_{i\in I} x_i$ is any ordering of the set $\{x_i | i \in I\}$.\\

\item Let $C_{n,0}$ denote the cylinder with $n$ marked points on one boundary component and no marked points on the other. Fix an orientation on the boundary component containing marked points and cyclically label them $1, \ldots, n$. Let $[i,j]$ denote the boundary segment $i\rightarrow j$. \\
\indent Note that $C_{n,0}$ arises as the partial triangulation of $\textnormal{M}_n$ consisting of a one-sided closed curve. We choose the canonical way of defining arcs on $C_{n,0}$.

\item Let $\gamma$ be an arc of $C_{n,0}$ with endpoints $i,j$. If $\gamma$ encloses a cylinder with boundary $[j,i]\cup {\gamma}$ then $\gamma := <i,j>$. If $\gamma$ encloses a cylinder with boundary $[i,j]\cup {\gamma}$ then $\gamma := <j,i>$, see Figure \ref{fig:pictureofarcs}.

\begin{figure}[H]
\centering
\includegraphics[width=80mm]{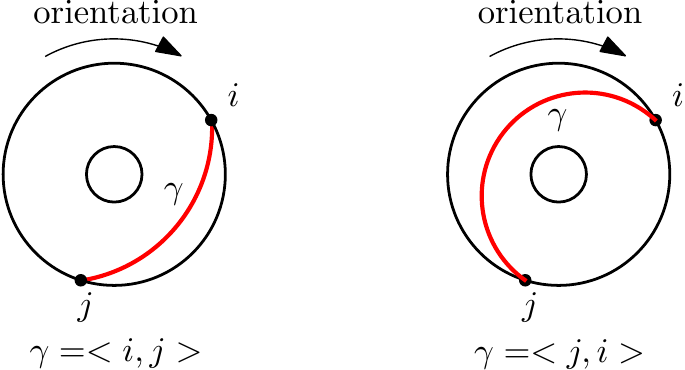}
\caption{Notation for an arc $\gamma$ of $C_{n,0}$}
\label{fig:pictureofarcs}
\end{figure}

\end{itemize}

The following theorem provides a very useful application of shellability.

\begin{thm}[Danaraj and Klee,\cite{shelling}]
\label{Ziegler}
Let $\Delta$ be a simplicial complex of dimension $n$. If $\Delta$ is a shellable psuedomanifold without boundary, then it is a PL $n$-sphere.

\end{thm}

\subsection{Shellability of $Arc(C_{n,0})$.}

The following proposition will help to prove the shellability of $Arc(\textnormal{M}_n)$, and is introduced now to cement key ideas.

\begin{prop}
\label{shellingcn}
$Arc(C_{n,0})$ is shellable for $n \geq 1$.

\begin{proof}

Consider the collection of triangulations $T(C_{n,0}^1) \subseteq T(C_{n,0})$ containing a loop at vertex $1$. Note that by cutting along the loop we get the $(n+1)$-gon (and a copy of $C_{1,0}$) for $n \geq 2$. We will prove by induction on $n$ that $T(C_{n,0}^1)$ is shellable. For $n=1$ the set $T(C_{1,0}^1) = T(C_{1,0})$ is trivially shellable.  For $n=2$  if we cut along the loop we get the triangle and $C_{1,0}$ which are both trivially shellable, so indeed $T(C_{2,0}^1)$ is shellable by Proposition \ref{join}.\\

Let $Block(i)$ be the set consisting of all triangulations in $T(C_{n,0}^1)$ containing the triangle with vertices $(1,1,i)$ for some $i\in [2,n]$, see Figure \ref{cylinderblock(i)}. \\ \indent Note that $Block(i) \equiv \prod_{j=1}^2 T(C_{m_j,0}^1)$ for $m_j < n$. By induction on $n$, $Block(i)$ is therefore the product of shellable sets. Taking the product of sets of triangulations corresponds to taking the join of their induced simplicial complexes. So Proposition \ref{join} tells us that $Block(i)$ is shellable. Denote this shelling by $S(Block(i))$.

\begin{figure}[H]
\centering
\includegraphics[width=25mm]{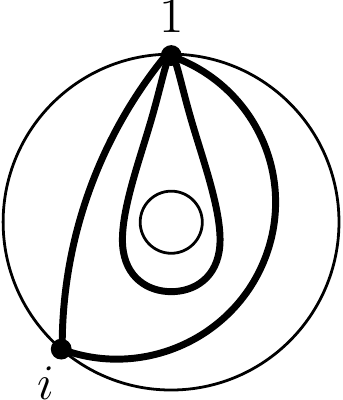}
\caption{$Block(i)$ consists of all triangulations of this partial triangulation.}
\label{cylinderblock(i)}
\end{figure}

\begin{claim}
\label{Claim 1}

The ordering $S(C_{n,0}^1) := \displaystyle\order_{i=n}^{2} S(Block(i))$ is a shelling for $T(C_{n,0}^1)$. \\

\begin{figure}[H]
\centering
\includegraphics[width=90mm]{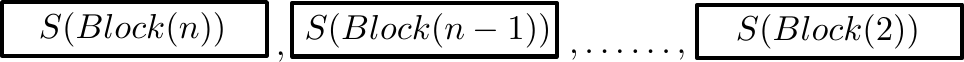}
\end{figure}

\begin{PC}

Let $S$ precede $T$ in the ordering $S(C_{n,0}^1)$. Then $T \in Block(k)$ and $S \in Block(j)$ for $j \geq k $. Since $S(Block(k))$ is a shelling for $Block(k)$ (by inductive assumption) then we may assume $j > k$. The arc $\gamma = <k,1> \in T$ is not compatible with the arc $<1,j> \in S$ so $\gamma \notin S$. Hence $T\cap S \subseteq T\cap \mu_{\gamma}(T)$. By Proposition \ref{Shelling} all that remains to show is that $\mu_{\gamma}(T)$ occurs before $T$ in the ordering. \\ Note that we will have a triangle in $T$ with vertices $(1,k,x)$ where $x \in [n,k+1]$. And so $\mu_{\gamma}(T) \in Block(x)$. Since $x>k$, $\mu_{\gamma}(T)$ does precede $T$ in the ordering. See Figure \ref{fig:claim1}. Hence $T(C_{n,0}^1)$ is shellable. 

\begin{figure}[H]
\centering
\includegraphics[width=105mm]{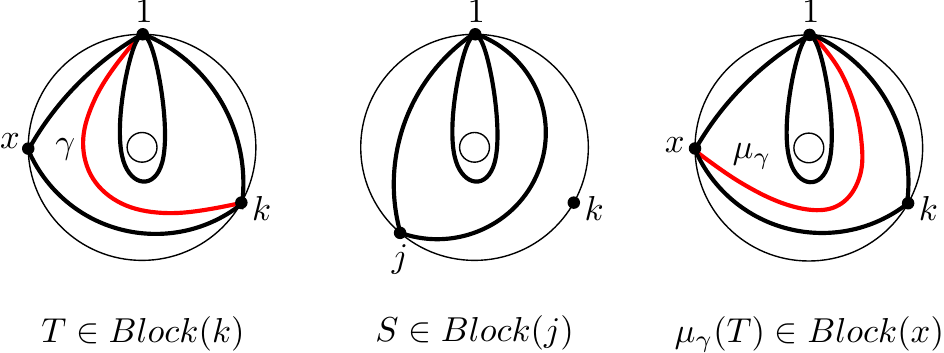}
\caption{}
\label{fig:claim1}
\end{figure}

\hfill \textit{End of proof of Claim \ref{Claim 1}.}

\end{PC}

\end{claim}

Similarly we can shell $T(C_{n,0}^i)$ in the same way $\forall i \in [1,n]$. Denote this shelling by $S(C_{n,0}^i)$

\begin{claim}
\label{Claim 2}

$S(Arc(C_{n,0})) := \displaystyle\order_{i=1}^{n} S(C_{n,0}^i)$ is a shelling for $Arc(C_{n,0})$

\begin{PC}

Let $S$ precede $T$ in the ordering $S(Arc(C_{n,0}))$. Then $T \in S(C_{n,0}^k)$ and $S \in S(C_{n,0}^j)$ for $1 \leq j \leq k$. Since $S(C_{n,0}^k)$ is a shelling we may assume $j<k$. There will be a triangle in $T$ with vertices $(k,k,x)$ for some $x \in [1,n]\backslash\{k\}$. \\

If $x \in [j,k-1]$ then mutate the loop at $k$ to give $T' \in S(C_{n,0}^x)$. $T'$ occurs before $T$ in the ordering because $x \in [j,k-1]$. Moreover since the loop at $k$ cannot occur in $S$ then $T\cap S \subseteq T\cap T'$. See Figure \ref{fig:claim2(1)}.

\begin{figure}[H]
\centering
\includegraphics[width=110mm]{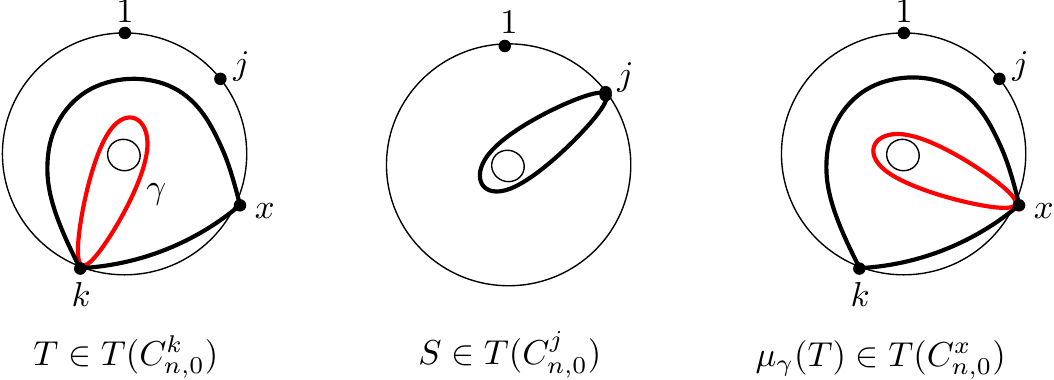}
\caption{Case when $x \in [j,k-1]$}
\label{fig:claim2(1)}
\end{figure}

If $x \in [k+1, j-1]$ then the arc $\gamma = <x,k>$ in $T$ is not compatible with the loop at $j$ in $S$. So $T\cap S \subseteq T\cap \mu_{\gamma}(T)$. Moreover the way we constructed the shelling $S(C_{n,0}^k)$ in Claim \ref{Claim 1} means that $\mu_{\gamma}(T)$ precedes $T$ in the ordering. See Figure \ref{fig:claim2(2)}.

\begin{figure}[H]
\centering
\includegraphics[width=110mm]{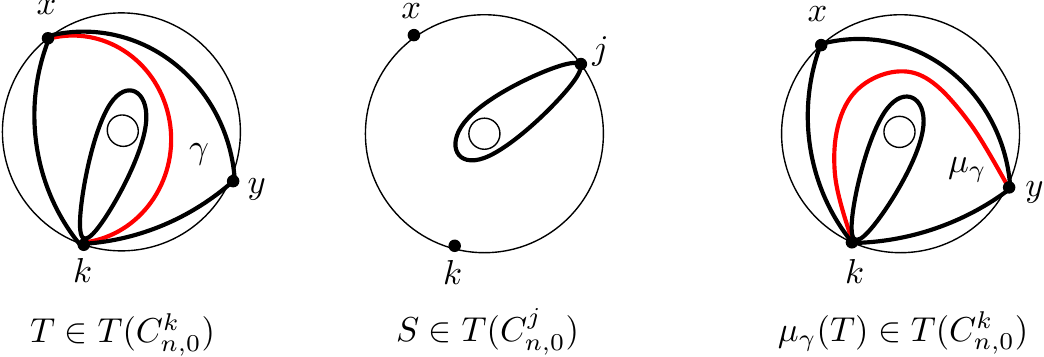}
\caption{Case when $x \in [k+1,j-1]$}
\label{fig:claim2(2)}
\end{figure}

\hfill \textit{End of proof of Claim \ref{Claim 2}.}

\end{PC}

\end{claim}

\end{proof}

\end{prop}

\begin{cor}
\label{n-gon}

$Arc(\text{$n$-gon})$ is shellable for $n\geq 3$

\begin{proof}

Follows immediately from Claim \ref{Claim 1}.

\end{proof}
\end{cor}

Applying Theorem \ref{Ziegler} we rediscover the classical result of Harer, \cite{harer2}.

\begin{cor}

$Arc(C_{n,0})$ and $Arc(\text{$n$-gon})$ are PL-spheres of dimension $n-2$ and $n-4$, respectively.

\end{cor}

\section{Main Theorem}
\label{main section}

In Section \ref{shellability} we achieved shellability of a complex by grouping facets into blocks and finding a `shelling order' in terms of these blocks. The task was then simplified to finding a shelling of the blocks themselves. Here we essentially follow the same strategy, twice. However, on the second iteration of the process we require a specific shelling of the blocks - in general an arbitrary shelling would not suffice.

\begin{defn}

Let $T(\textnormal{M}_n^\circ) \subseteq T(\textnormal{M}_n)$ consist of all triangulations of $\textnormal{M}_n$ (i.e no quasi-triangulations containing a one-sided curve).

\end{defn}

\begin{defn}

Let $\gamma$ be an arc in $T \in T(\textnormal{M}_n^\circ)$. Call $\gamma$ a \textit{\textbf{crosscap arc}} (c-arc) if $\textnormal{M}_n\setminus\{\gamma\}$ is orientable. (Informally, a c-arc is an arc that necessarily passes through the crosscap). Let $(i,j)$ denote a c-arc with endpoints $i$ and $j$.
\end{defn}. 

\begin{defn}

Call a triangulation $T \in T(\textnormal{M}_n^\circ)$ a \textit{\textbf{crosscap triangulation}} (c-triangulation) if every arc in $T$ is a c-arc. Let $T(\textnormal{M}_n^\otimes) \subseteq T(\textnormal{M}_n^\circ)$ consist of all c-triangulations.
\end{defn}

\begin{defn}

Let $\gamma$ be an arc in $T \in T(\textnormal{M}_n^\circ)$ that is not a c-arc. Call $\gamma$ a \textit{\textbf{bounding arc}} (b-arc) if it mutates to a c-arc.

\end{defn}

\begin{figure}[H]
\centering
\includegraphics[width=120mm]{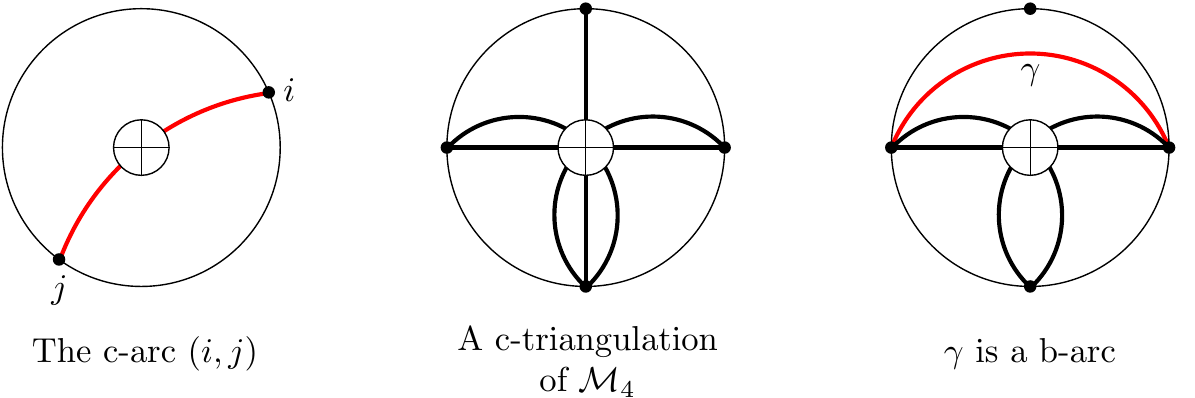}
\caption{}
\end{figure}

\subsection{Reducing the problem to c-triangulations.}

\begin{lem}
\label{half}
If $T(\textnormal{M}_n^\otimes)$ is shellable then so is $T(\textnormal{M}_n^\circ)$.

\begin{proof}

Consider $I := \{i_1, \ldots, i_k\} \subseteq [1,n]$. Let $\Gamma_{I}^{(k)}$ consist of all triangulations $T \in T(\textnormal{M}_n^\circ)$ such that there is a c-arc in $T$ with endpoint $j$ \textit{if and only if} $j \in I$. Note that this condition implies the existence of an arc or boundary segment $<i_m,i_{m+1}>$ (where $i_{k+1}:= i_1$) in every triangulation $T \in \Gamma_{I}^{(k)}$ $\forall m \in [1,k]$.

\begin{figure}[H]
\centering
\includegraphics[width=110mm]{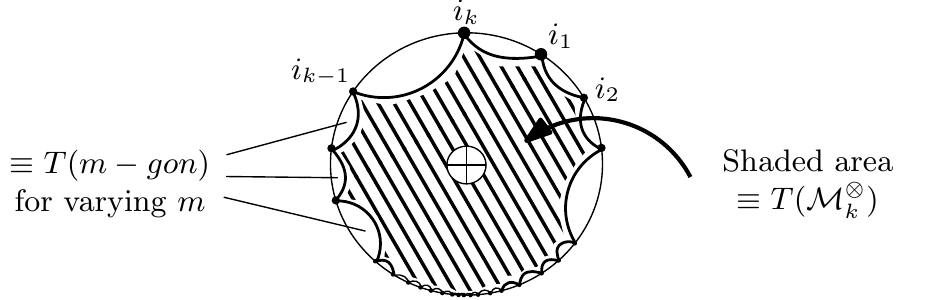}
\caption{$\Gamma_I^{(k)}$}
\end{figure}

By assumption $T(\textnormal{M}_n^\otimes)$ is shellable, and by Corollary \ref{n-gon} $T(\text{$m$-gon})$ is also shellable. Hence $\Gamma_{I}^{(k)}$ is the product of shellable collections of triangulations, and so is shellable by Proposition \ref{join}. Denote this shelling by $S(\Gamma_{I}^{(k)})$.

\begin{claim}
\label{Claim 3}

Let $Block(k) := \displaystyle\order_{I \in [1,n]^{(k)}}S(\Gamma_{I}^{(k)})$. Then $\displaystyle\order_{k=n}^{1} Block(k)$ is a shelling for $T(\textnormal{M}_n^\circ)$. \\

\begin{PC}

Let $S$ precede $T$ in the ordering. Then $S \in Block(j)$ and $T \in Block(k)$ where $j \geq k$. In particular, $T \in S(\Gamma_{I_1}^{(k)})$ and $S \in S(\Gamma_{I_2}^{(j)})$ for some $I_1,I_2 \in \mathcal{P}([1,n])$ where $|I_1| \leq |I_2|$. Since $S(\Gamma_{I}^{(k)})$ is a shelling we may assume $I_1 \neq I_2$. \\

Suppose that every b-arc in $T$ is also an arc in $S$. Then $I_2 \subseteq I_1$, and since $|I_1| \leq |I_2|$ this implies $I_1 = I_2$.
So we may assume there is at least one b-arc $\gamma \in T$ that is not an arc in $S$. Since $\gamma \notin S$, $T\cap S \subseteq T\cap \mu_{\gamma}(T)$. Moreover, since $\gamma$ is a b-arc, $\mu_{\gamma}(T) \in Block(k+1)$. Hence $\mu_{\gamma}(T)$ precedes $T$ in the ordering, see Figure \ref{fig:claim3}. \\

\begin{figure}[H]
\centering
\includegraphics[width=140mm]{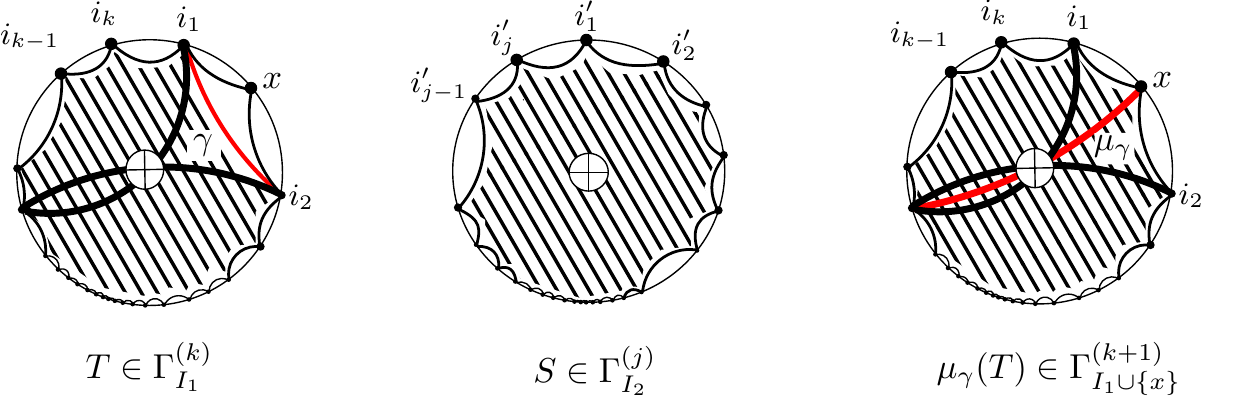}
\caption{}
\label{fig:claim3}
\end{figure}

\hfill \textit{End of proof of Claim \ref{Claim 3}.}
\end{PC}

\end{claim}

\end{proof}

\end{lem}

The idea behind Lemma \ref{half} is that we are decomposing $T(\text{M}^{\circ}_n)$ into blocks, and ordering these blocks. The ordering is chosen in such a way that if we manage to individually shell the blocks themselves, we'll have a shelling of $T(\text{M}^{\circ}_n)$. Figure \ref{fig:colouredM3} shows the block structure of $T(\text{M}^{\circ}_3)$.

In particular, we realise that to shell a block it is sufficient to find a shelling of $T(\textnormal{M}_n^\otimes)$. We will split this into two cases: $n$ even and $n$ odd.

\begin{figure}[H]
\centering
\includegraphics[width=140mm]{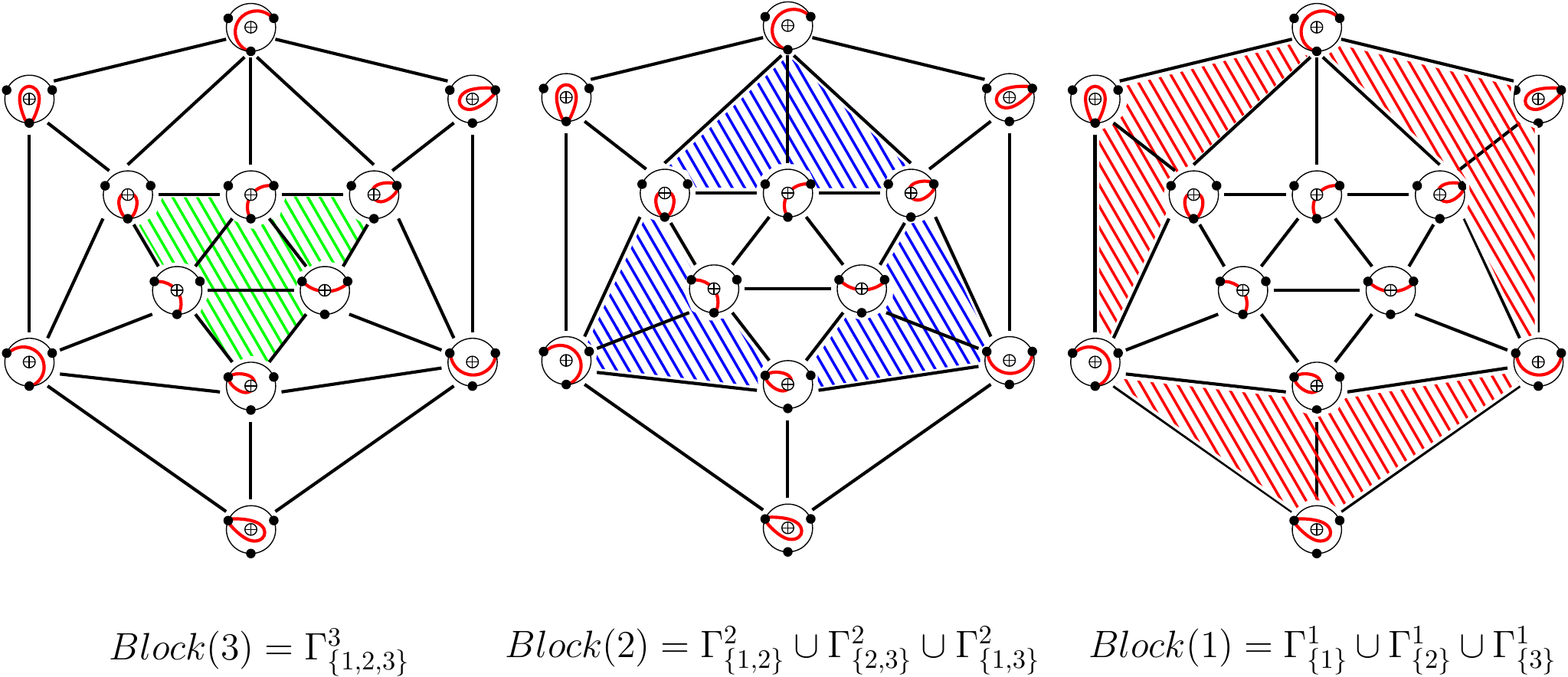}
\caption{Block structure of $T(\text{M}^{\circ}_3)$}
\label{fig:colouredM3}
\end{figure}

\subsection{Shellability of $T(\textnormal{M}_n^{\otimes})$ for even $n$.}

Let $D_{\{(1,\frac{n}{2}+1)\}}^n$ consist of all triangulations of $T(\textnormal{M}_n^{\otimes})$ containing the c-arc $(1, \frac{n}{2}+1)$ but containing no other c-arcs $(i,\frac{n}{2}+i)$ $\forall i \in [2,n]$. See Figure \ref{fig:onediagonal}.

\begin{figure}[H]
\centering
\includegraphics[width=120mm]{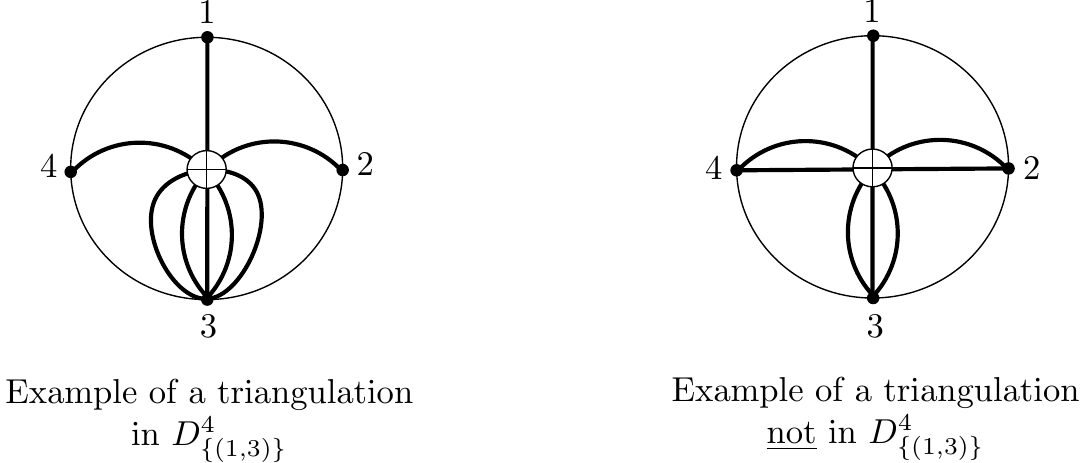}
\caption{}
\label{fig:onediagonal}
\end{figure}

\begin{defn}

Let $T \in D_{\{(1,\frac{n}{2}+1)\}}^n$ and $\gamma$ a c-arc in $T$. $\gamma = (i,j)$ for some $i \in [1, 1+\frac{n}{2}]$ and $j \in [1+\frac{n}{2},1]$. Define the \textit{\textbf{length}} of $\gamma$ as follows: 

\begin{itemize}

\item If $i=j=1$, $l(\gamma) := n+1$.
\item Otherwise, $l(\gamma):= |[i,j]|$.

\end{itemize}

\begin{figure}[H]
\centering
\captionsetup{justification=centering,margin=2cm}
\includegraphics[width=35mm]{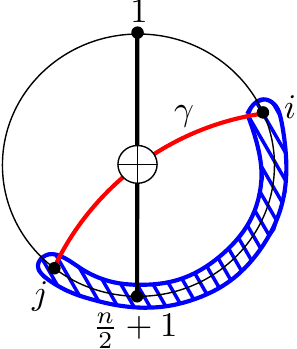}
\caption{If $i \neq 1$ or $j \neq 1$ then the number of marked points in the shaded tube equals $l(\gamma)$.}

\end{figure}

\end{defn}

\begin{defn}

Let $\mathcal{X}_1^n$ be the partial triangulation of $\textnormal{M}_n$ consisting of the c-arcs $(1, \frac{n}{2}+1), (2,\frac{n}{2}+1), (n,\frac{n}{2}+1)$. Additionally, let $T(\mathcal{X}_1^n)$ denote the triangulations in $D_{\{(1,\frac{n}{2}+1)\}}^n$ containing the c-arcs $(1, \frac{n}{2}+1), (2,\frac{n}{2}+1), (n,\frac{n}{2}+1)$. \\
Similarly, let $\mathcal{X}_2^n$ be the partial triangulation of $\textnormal{M}_n$ consisting of the c-arcs $(1, \frac{n}{2}+1), (1,\frac{n}{2}), (n,\frac{n}{2}+2)$. Let $T(\mathcal{X}_1^n)$ denote the triangulations in $D_{\{(1,\frac{n}{2}+1)\}}^n$ containing the c-arcs $(1, \frac{n}{2}+1), (2,\frac{n}{2}+1), (n,\frac{n}{2}+1)$. See Figure \ref{fig:typex}. \\

\begin{figure}[H]
\centering
\includegraphics[width=90mm]{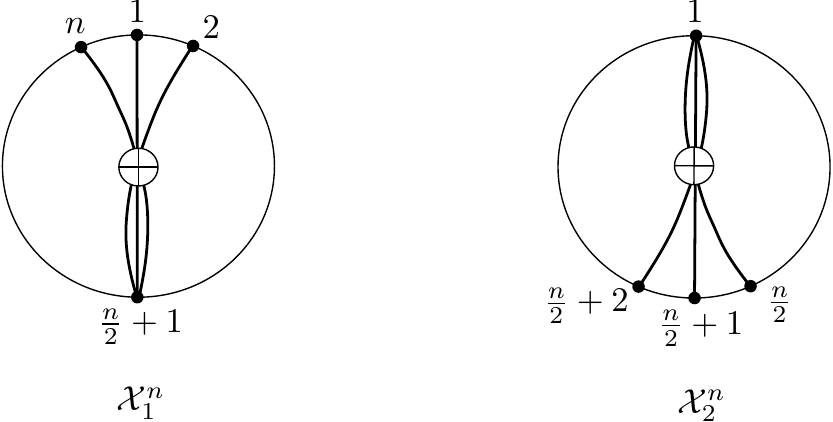}
\caption{}
\label{fig:typex}
\end{figure}

\end{defn}

\begin{lem}
\label{type}
$D_{\{(1,\frac{n}{2}+1)\}}^n = T(\mathcal{X}_1^n) \bigsqcup T(\mathcal{X}_2^n)$. Moreover, for any c-arc $\gamma \neq (1, \frac{n}{2}+1)$ in $T$ we have the following:

\begin{itemize} 
\item $l(\gamma) \leq \frac{n}{2}$ if $T \in T(\mathcal{X}_1^n)$.
\item $l(\gamma) \geq \frac{n}{2}+2$ if $T \in T(\mathcal{X}_2^n)$.

\end{itemize} 

\begin{proof}

A triangulation $T$ in $D_{\{(1,\frac{n}{2}+1)\}}^n$ will contain either the c-arc $(2,\frac{n}{2}+1)$ or the c-arc $(1,\frac{n}{2}+2)$.\\

Assume the c-arc $(2,\frac{n}{2}+1)$ is in $T$. We will show, by induction on $i$, the c-arc of maximal length in $T$ with endpoint $i \in [2,\frac{n}{2}+1]$ must be the c-arc $(i,x)$ where $x \in [\frac{n}{2}+1,\frac{n}{2}+i-1]$. \\

Let $\gamma$ be the c-arc in $T$ of maximal length with endpoint $2$. Let $j$ be the other endpoint of $\gamma$ and suppose for a contradiction $j \in [\frac{n}{2}+2,n]$. Since $(2,\frac{n}{2}+1) \in T$ then, as T is a c-triangulation, $(2, x) \in T$ $\forall x \in [\frac{n}{2}+1,j]$. In particular $\beta := (2,\frac{n}{2}+2) \in T$ - which contradicts $T \in D_{\{(1,\frac{n}{2}+1)\}}^n$. See Figure \ref{fig:inductioni2}.\\

\begin{figure}[H]
\centering
\includegraphics[width=35mm]{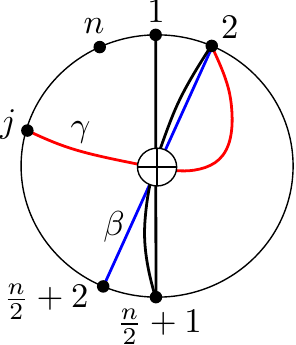}
\caption{}
\label{fig:inductioni2}
\end{figure}

By induction, the c-arc $\alpha$ of maximal length in $T$ with endpoint $i-1$ is the c-arc $(i-1, x)$ where $x \in [\frac{n}{2}+1,\frac{n}{2}+i-2]$. Let $\gamma$ be the c-arc in $T$ of maximal length with endpoint $i$. Let $j$ be the other endpoint of $\gamma$ and suppose $j \in [\frac{n}{2}+i,n]$. But by the maximality of $\alpha$ there will be a c-arc $(i,y)$ $\forall y \in [x,j]$. In particular there will be a c-arc $\beta := (i,\frac{n}{2}+i)$. See Figure \ref{fig:maxlength}.\\

\begin{figure}[H]
\centering
\includegraphics[width=55mm]{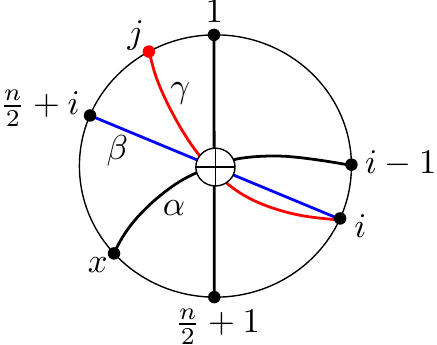}
\caption{}
\label{fig:maxlength}
\end{figure}

If we supposed $(1,\frac{n}{2}+2)$ was an arc in $T$, then an analogous argument shows that $T \in T(\mathcal{X}_2)$.

\end{proof}

\begin{cor}
\label{intersection}

Let $S \in T(\mathcal{X}_1^n)$ and $T \in T(\mathcal{X}_2^n)$ then $S\cap T = \{(1, \frac{n}{2}+1)\}$

\begin{proof}

It follows from the fact that, excluding the c-arc $(1, \frac{n}{2}+1)$, the maximal length of any c-arc in $\mathcal{X}_1^n$ is less than or equal to $ \frac{n}{2}$, and the minimal length of any c-arc in $\mathcal{X}_2^n$ is greater than or equal to $ \frac{n}{2}+2$.

\end{proof}
\end{cor}

\begin{cor} \mbox{}
\label{tmax}
The triangulation $T_{\text{max}}$ in Figure \ref{fig:tmax} is the unique triangulation in $T(\mathcal{X}_1^n)$ such that $\sum\limits_{\gamma \in T_{\text{max}}} l(\gamma)$ is maximal. 
The triangulation $T_{\text{min}}$ is the unique triangulation in $T(\mathcal{X}_2^n)$ such that $\sum\limits_{\gamma \in T_{\text{max}}} l(\gamma)$ is minimal. More explicitly, \newline
\noindent $T_{\text{max}} := \{(1, \frac{n}{2}+1)\} \cup \{(i,\frac{n}{2}+i-1)| i \in [2,\frac{n}{2}+1]\} \cup \{(i,\frac{n}{2}+i-2)| i \in [3,\frac{n}{2}+1]\}$.

\noindent $T_{\text{min}} :=  \{(1, \frac{n}{2}+1)\} \cup \{(i,\frac{n}{2}+i+1)| i \in [1,\frac{n}{2}]\} \cup \{(i,\frac{n}{2}+i+2)| i \in [1,\frac{n}{2}-1]\}$.
\begin{proof}

Consider the partial triangulation $\mathcal{P}$ of $\mathcal{X}_1^n$ consisting of all the c-arcs of maximal length. Namely the c-arcs $(i,\frac{n}{2}+i-1)$ $\forall i \in [2,\frac{n}{2}+1]$. $\mathcal{P}$ cuts $\textnormal{M}_n$ into ($2$ triangles and) quadrilaterals bounded by the two boundary segments $[i,i+1],[\frac{n}{2}+i-1,\frac{n}{2}+i]$ and the two c-arcs $(i,\frac{n}{2}+i-1), (i+1,\frac{n}{2}+i)$ $\forall i \in [3,\frac{n}{2}]$. Let $T$ be a triangulation of $\mathcal{P}$ such that $T \in T(\mathcal{X}_1^n)$. Notice that $(i,\frac{n}{2}+i) \notin T$ by definition of $D_{\{(1,\frac{n}{2}+1)\}}^n$, hence $(i+1,\frac{n}{2}+i-1) \in T$ $\forall i \in [3,\frac{n}{2}+1]\}$ and so $T = T_{\text{max}}$. Moreover, since $l(i+1,\frac{n}{2}+i-1) = l(i,\frac{n}{2}+i-1) - 1$ then $T$ is the unique triangulation in $T(\mathcal{X}_1^n)$ such that $\sum\limits_{\gamma \in T} l(\gamma)$ is maximal. \\

\noindent Analogously we get the result regarding unique minimality of $T_{\text{min}}$.

\begin{figure}[H]
\centering
\includegraphics[width=115mm]{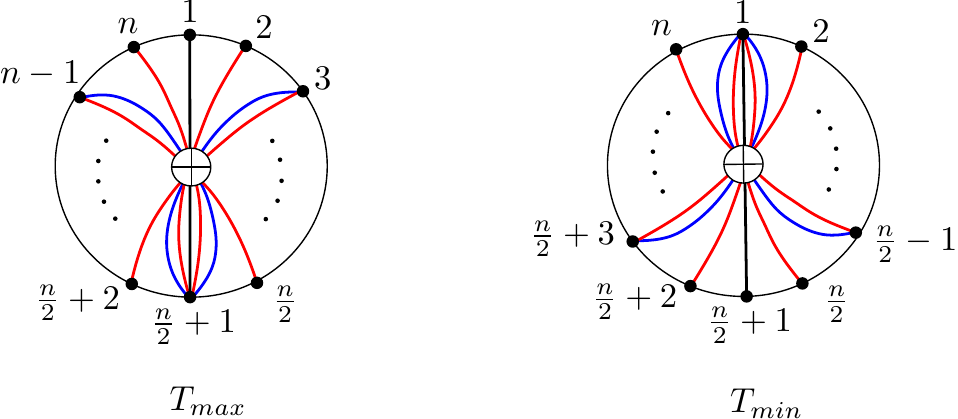}
\caption{}
\label{fig:tmax}
\end{figure}

\end{proof}

\end{cor}

\end{lem}

\begin{defn}

Call a c-arc $(i,\frac{n}{2}+i)$ of $\textnormal{M}_n$ a \textit{\textbf{diagonal}} arc. 

\end{defn}

\begin{defn}

Consider a c-arc $\gamma$ in a triangulation of $\mathcal{X}_1^n$. If $l(\gamma) = \frac{n}{2}$ then call $\gamma$ a \textit{\textbf{max}} arc.

\end{defn}

\begin{defn}

Consider a c-arc $\gamma$ in a triangulation of $\mathcal{X}_2^n$. If $l(\gamma) = \frac{n}{2}+2$ then call $\gamma$ a \textit{\textbf{min}} arc.

\end{defn}

Consider a partial triangulation of $\mathcal{X}_1^n$ containing two max arcs. Cutting along these max arcs we will be left with two regions. Let $R$ be the region that doesn't contain the diagonal arc $(1,\frac{n}{2}+1)$. Note $R$ will contain $2k$ marked points for some $k \in \{2,\ldots, \frac{n}{2}\}$.

\begin{lem}
\label{R}
The set of triangulations of $R$ such that no max arcs occur in $R$ is equivalent to $T(\mathcal{X}_1^{2(k-1)})$.

\begin{proof}

Collapse the quadrilateral $(1,2, \frac{n}{2}+1,n)$ to a c-arc and relabel marked points as shown in Figure \ref{fig:collapse}.

\begin{figure}[H]
\centering
\includegraphics[width=140mm]{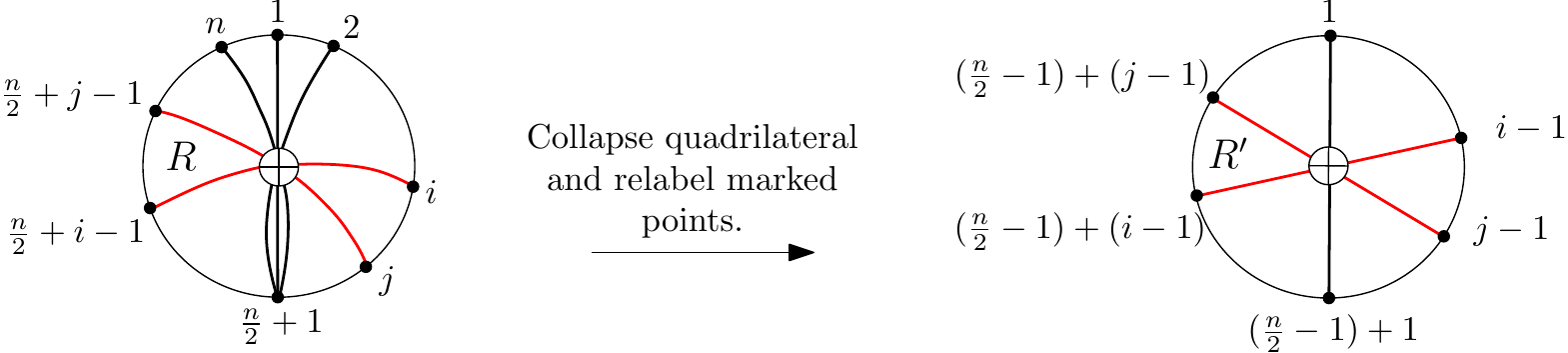}
\caption{}
\label{fig:collapse}
\end{figure}

Max arcs in $R$ correspond to diagonal arcs in $R'$. Furthermore, up to a relabelling of vertices, triangulating $R'$ so that no diagonal arcs occur in the triangulation is precisely triangulating $\mathcal{X}_1^{2(k-1)}$ so that no diagonal arcs occur.

\end{proof}

\end{lem}

\begin{rmk}
\label{dyck remark}
Using induction we realise that Lemma \ref{R} tells us that $D_{\{(1,\frac{n}{2}+1)\}}^n$ has the same flip structure as the set of all Dyck paths of length $n-2$. In particular, triangulations in $D_{\{(1,\frac{n}{2}+1)\}}^n$  correspond to Dyck paths, and arcs appearing in those triangulations correspond to nodes in the Dyck lattice. This correspondence is indicated in Figure \ref{fig:dyck} and is best viewed in colour.

\end{rmk}

\begin{sidewaysfigure}
\begin{subfigure}{0.55\hsize}\centering
    \includegraphics[width=0.9\hsize]{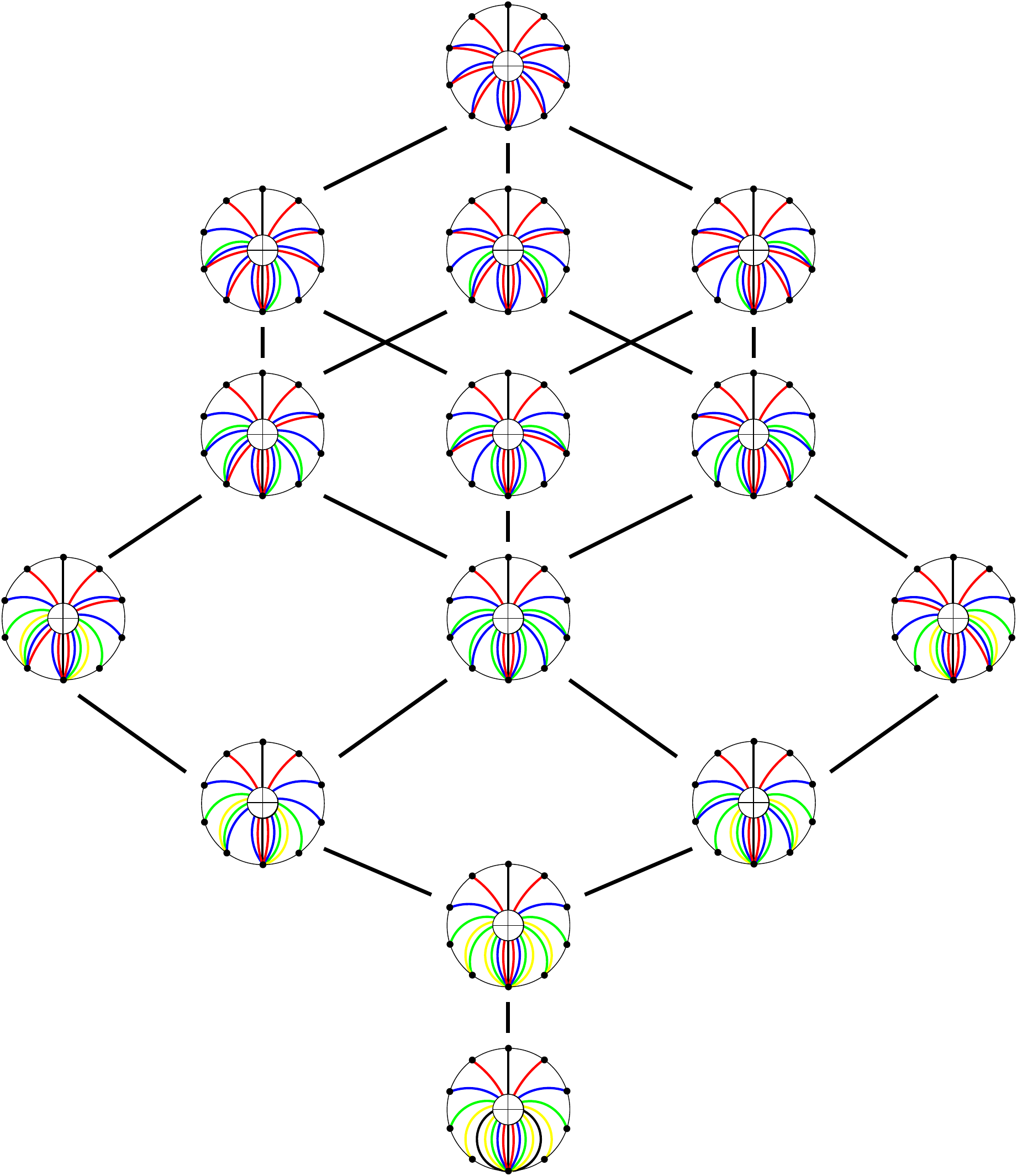}
\end{subfigure}%
\hfill
\begin{subfigure}{0.55\hsize}\centering
    \includegraphics[width=0.9\hsize]{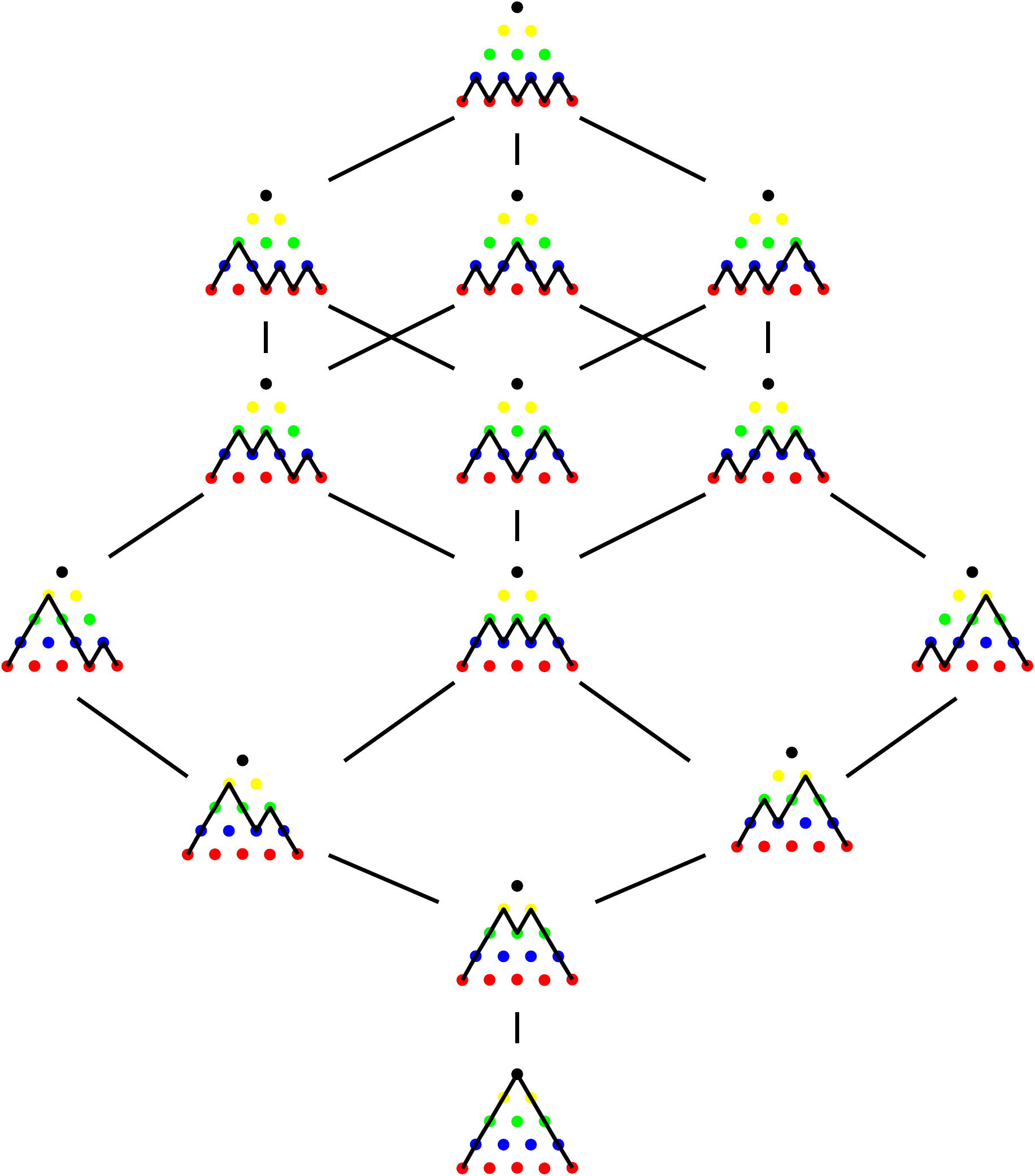}
\end{subfigure}
\caption{$D_{\{(1,\frac{n}{2}+1)\}}^n$ and Dyck paths. See Remark \ref{dyck remark} for an explanation of their connection.}
\label{fig:dyck}
\end{sidewaysfigure}

\begin{defn}

Let $i \in \{1,2\}$. Call an arc $\gamma$ in $T \in T(\mathcal{X}_i^n)$ \textit{\textbf{$\mathcal{X}$-mutable}} if $\mu_{\gamma}(T) \in T(\mathcal{X}_i^n)$.

\end{defn}

\begin{defn}

Let $\gamma$ be an $\mathcal{X}$-mutable arc in a triangulation $T \in D_{\{(1,\frac{n}{2}+1)\}}^n$, and let $\gamma'$ be the arc $\gamma$ mutates to. Call $\gamma$ \textit{\textbf{upper-mutable}} if $l(\gamma') > l(\gamma)$ and \textit{\textbf{lower-mutable}} if $l(\gamma') < l(\gamma)$.

\end{defn}

\begin{defn}

Call a shelling $\mathcal{S}$ of $T(\mathcal{X}_1^n)$ $(T(\mathcal{X}_2^n))$ an \textit{\textbf{upper (lower)}} shelling if for any triangulation $T \in \mathcal{S}$ and any upper (lower) mutable arc $\gamma$ in $T$, $\mu_{\gamma}(T)$ precedes $T$ in the ordering.

\end{defn}

\begin{defn}

Let $\mathcal{I}$ be the set of all max arcs of $D_{\{(1,\frac{n}{2}+1)\}}^n$, excluding the max arcs $\alpha_1 :=(1,\frac{n}{2}+1), \alpha_2 :=(\frac{n}{2}+1,n)$.

\end{defn}

\begin{lem}
\label{no max}
If $T \in T(\mathcal{X}_1^n)$ doesn't contain a max arc $m \in \mathcal{I}$ then there exists an upper mutable arc $\gamma$ strictly contained between the endpoints of $m$, see Figure \ref{fig:nom}.

\begin{figure}[H]
\centering
\includegraphics[width=45mm]{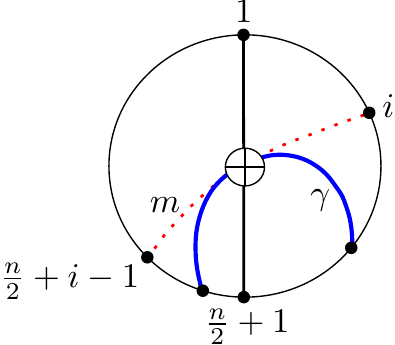}
\caption{}
\label{fig:nom}
\end{figure}

\begin{proof}

If $n\in\{2,4\}$ then $\mathcal{I} = \emptyset$  and there is nothing to prove. So assume $n \geq 6$. \\

Suppose $m = (i,\frac{n}{2}+i -1)\in \mathcal{I}$ is not in the triangulation $T$. We will show there exists a c-arc strictly contained between the endpoints of $m$. \\

Let $(i,x)$ be the c-arc of maximum length in $T$ connected to $i$. Since $m \neq (i,x)$ then $x \in [\frac{n}{2}+1,\frac{n}{2}+i-2]$. Moreover, by maximality of $(i,x)$, $(i+1,x) \in T$. So indeed there is a c-arc in $T$ strictly contained between the endpoints of $m$, see Figure \ref{fig:nomax2}.

\begin{figure}[H]
\centering
\includegraphics[width=55mm]{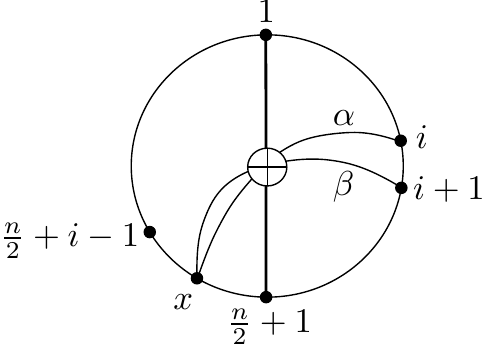}
\caption{}
\label{fig:nomax2}
\end{figure}

Of the c-arcs that are strictly contained between the endpoints of $m$, let $\gamma = (j_1,j_2)$ be an arc of minimum length. We will show that $\gamma$ is upper mutable. \\

By minimality of $\gamma$ the c-arc $(j_1,j_2-1)$ is not in $T$. Hence the c-arc $(j_1-1,j_2)$ must be in $T$. Likewise the c-arc $(j_1,j_2+1)\in T$. So $\gamma$ is contained in the quadrilateral $(j_1,j_1-1,j_2,j_2+1)$. Hence mutating $\gamma$ gives $\gamma' = (j_1-1,j_2+1)$. $l(\gamma) < l(\gamma')$ so $\gamma$ is indeed upper mutable, see Figure \ref{fig:nomax3}.

\begin{figure}[H]
\centering
\includegraphics[width=55mm]{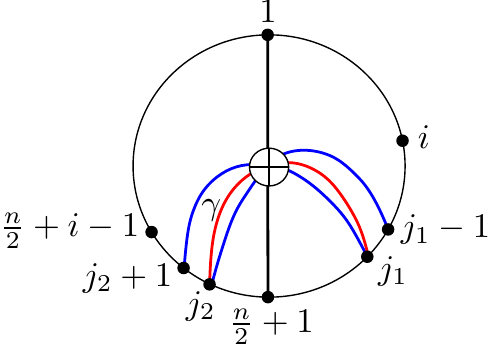}
\caption{}
\label{fig:nomax3}
\end{figure}

\end{proof}

\end{lem}

\begin{lem}
\label{upper shelling}
There exists an upper shelling for $T(\mathcal{X}_1^n)$. Denote this by $S(\mathcal{X}_1^n)$.

\begin{proof}

Let $\Psi_{\{\gamma_1,\ldots ,\gamma_k\}}$ be the collection of triangulations in $T(\mathcal{X}_1^n)$ containing the max arcs $\gamma_1,\ldots ,\gamma_k, \alpha_1, \alpha_2$ and no other max arcs. By Lemma \ref{R} we know that $\Psi_{\{\gamma_1,\ldots ,\gamma_k\}} \equiv \prod_{i=1}^{j}$ T($\mathcal{X}_1^{m_i})$. \\
Moreover, by induction on the trivial base case when $n=2$, and using Proposition \ref{join}, we get that there is an upper shelling for $\Psi_{\{\gamma_1,\ldots ,\gamma_k\}}$. Denote this shelling by $S(\Psi_{\{\gamma_1,\ldots ,\gamma_k\}})$. \\

\begin{claim}
\label{block}
Let $Block(k) := \displaystyle\order_{J \in \mathcal{I}^{(k)}}S(\Psi_J)$. Then $\displaystyle\order_{k=\frac{n}{2}-2}^{0} Block(k)$ is an upper shelling for $T(\mathcal{X}_1^n)$.

\begin{PC}

Let $T,S \in T(\mathcal{X}_1^n)$ and suppose $S$ precedes $T$ in the proposed ordering. Then $T \in \Psi_{J_1}$ and $S \in \Psi_{J_2}$ where $J_1,J_2 \in \mathcal{P}([1,n])$ and $|J_1| \leq |J_2| $. \textit{W.l.o.g.} we may assume $J_1 \neq J_2$ since by induction $S(\Psi_{J_1})$ is an upper shelling. \\

As $|J_1| \leq |J_2| $ and $J_1 \neq J_2 $ there is a max arc $m$ in $S$ that is not in $T$. By Lemma \ref{no max} there is an upper mutable arc $\gamma$ in $T$ strictly contained between the endpoints of $m$. Moreover $\gamma$ and $m$ are not compatible so $S\cap T \subseteq \mu_{\gamma}(T)\cap T$. And $\mu_{\gamma}(T)$ precedes $T$ in the ordering because of the upper shelling $S(\Psi_{J_1})$.

\end{PC}

\hfill \textit{End of proof of Claim \ref{block}.}

\end{claim}

\end{proof}

\end{lem}

An analogous argument proves the following lemma.

\begin{lem}
There exists a lower shelling for $T(\mathcal{X}_2^n)$. Denote this by $S(\mathcal{X}_2^n)$.

\end{lem}

\begin{defn}

Call a c-arc $\gamma$ in a triangulation $T \in D_{\{(1,\frac{n}{2}+1)\}}^n$ \textit{\textbf{special mutable}} if any of the following is true:

\begin{itemize}

\item $T \in T(\mathcal{X}_1^n)$ and $\gamma$ is upper mutable.
\item $T \in T(\mathcal{X}_2^n)$ and $\gamma$ is lower mutable.
\item $\gamma$ mutates to a diagonal c-arc.

\end{itemize}

\end{defn}

\begin{lem}
\label{counted}
For any $T \in T(\mathcal{X}_1^n)\setminus \{T_{max}\}$, $T_{max}$ is connected to $T$ by a sequence of lower mutations.

\begin{proof}

By Lemma \ref{no max} we can keep performing mutations on upper mutable arcs until we reach a triangulation containing every max arc. By Corollary \ref{tmax} the only triangulation in $T(\mathcal{X}_1^n)$ that contains every max arc is $T_{max}$. Hence $T$ is connected to $T_{max}$ by a sequence of upper mutations. Equivalently, $T_{max}$ is connected to $T$ by a sequence of lower mutations.

\end{proof}

\end{lem}

\begin{lem}
\label{fix even}
Let $T \in D_{\{(1,\frac{n}{2}+1)\}}^n$ and let $P_T$ be the partial triangulation of $\textnormal{M}_n$ consisting of all the special mutable arcs in $T$. Then any triangulation of $P_T$ cannot contain the diagonal c-arc $(i,\frac{n}{2}+i)$ $\forall i \in \{2,\ldots,\frac{n}{2}\}$.

\begin{proof} Assume $T \in T(\mathcal{X}_1^n)$. An analogous argument works if $T \in T(\mathcal{X}_2^n)$. We prove the lemma via induction on the upper shelling order of $T(\mathcal{X}_1^n)$. \\

The first triangulation in the upper shelling ordering is $T_{max}$. The special mutable arcs in $T_{max}$ are $(i, \frac{n}{2}+i-2)$ $\forall i \in [3, \frac{n}{2}+1]$. However, the c-arc $(i, \frac{n}{2}+i-2)$ is not compatible with the diagonal c-arc $(i-1, \frac{n}{2}+i-1)$. And so ranging $i$ over $3,\ldots, \frac{n}{2}+1$ proves the base inductive case. \\

Let $\gamma$ be a lower mutable arc in a triangulation $T \in T(\mathcal{X}_1^n)$. By Lemma \ref{counted}, to prove the lemma it suffices to show that the special mutable arcs in $\mu_{\gamma}(T)$ prevent the same diagonal c-arcs as the special mutable arcs in $T$. Let $\beta_1 , \beta_2$ be the c-arcs containing $\gamma$ in a quadrilateral. See Figure \ref{fig:fixspecial}.

\begin{figure}[H]
\centering
\includegraphics[width=35mm]{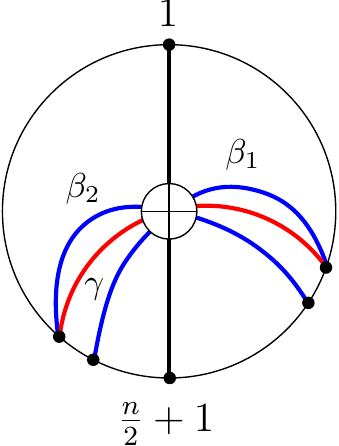}
\caption{}
\label{fig:fixspecial}
\end{figure}

The arcs $\beta_1$ and $\beta_2$ may be special mutable in $T$ but in $\mu_{\gamma}(T)$ they definitely won't be. The implication of this is that $\beta_1$ and $\beta_2$ may be c-arcs in $P_T$, and prevent certain diagonal arcs, but $\beta_1,\beta_2 \notin P_{\mu_{\gamma}(T)}$ so $\mu_{\gamma}$ needs to make up this difference. Indeed, it does make up the difference as the diagonal arcs not compatible with either $\beta_1$ or $\beta_2$ are precisely the diagonal arcs not compatible with $\mu_{\gamma}$.

\end{proof}

\end{lem}

\begin{lem}
\label{at least one diagonal}
In each c-triangulation $T$ of $\textnormal{M}_n$ there is at least one diagonal arc.

\end{lem}

\begin{proof}

Let us assume, for a contradiction, that there is no diagonal arc in $T$. Without loss of generality, we may assume that the c-arc connected to $1$, of maximum length, is $\gamma = (1,j_1)$ for some $j_1 \in [1,\frac{n}{2}]$. (Otherwise just flip the picture.)

Let $\gamma_2 = (2,j_2)$ be the c-arc of maximum length in $T$ that is connected to $2$. If $j_2 > \frac{n}{2}$ then by maximality of $\gamma_1$ there is a c-arc $(2,\frac{n}{2})$. Hence, $j_2 \in [j_1,\frac{n}{2}+1]$. Inductive reasoning shows that the c-arc connected to $j_1-1$ in $T$, of maximum length, is $\gamma_{j-1}=(j-1,x)$ for some $x \in [j, \frac{n}{2}+j_1-2]$. However, then by the maximality of $\gamma_{j-1}$ we must have $(j_1, \frac{n}{2}+j_1) \in T$. This gives a contradiction, and so the lemma is proved.

\end{proof}

\begin{lem}
\label{even}
 $T(\textnormal{M}_n^{\otimes})$ is shellable for even $n$. 
 
\end{lem}

\begin{proof}

Let $\mathcal{K}$ be the collection of diagonal c-arcs of $\textnormal{M}_n$. Consider $I = \{ \gamma_1, \ldots, \gamma_k \} \subseteq \mathcal{K}$ and let $D_{I}^n$ consist of all triangulations of $T(\textnormal{M}_n^{\otimes})$ containing every diagonal c-arcs in $I$, and no diagonal c-arcs in $\mathcal{K}\setminus I$. The set of c-triangulations $T(R)$ of a region $R$ cut out by two diagonal c-arcs, so that no other diagonal c-arcs occur in the region, is equivalent to $D_{\{(1,\frac{m}{2}+1)\}}^m$ for some $m \in [2,n-2]$. See Figure \ref{fig:collapsediagonal}.

\begin{figure}[H]
\centering
\includegraphics[width=120mm]{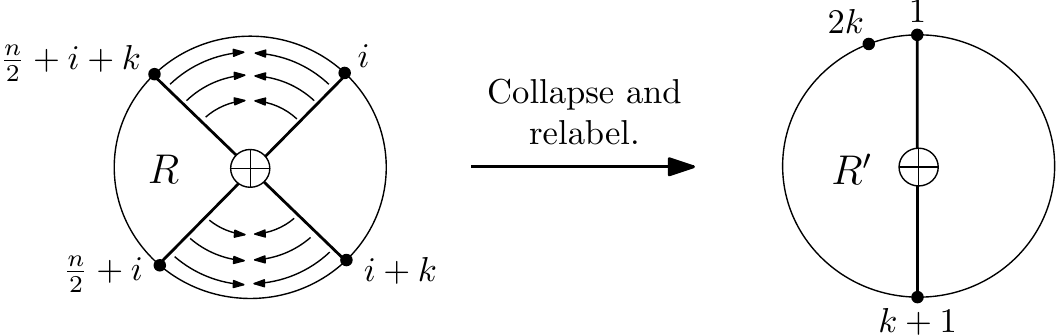}
\caption{$T(R) \equiv T(R') = D_{\{(1,k+1)\}}^{2k}$}
\label{fig:collapsediagonal}
\end{figure} 

Choose $\displaystyle\order_{i=1}^2 S(\mathcal{X}_i^m)$ to be the ordering of $D_{\{(1,\frac{m}{2}+1)\}}^m$. Take the disjoint union of these orderings, over all the regions cut out by diagonal c-arcs in $I$, to get an ordering of $D_{I}^n$. Denote this ordering by $O(D_{I}^n)$.

\begin{claim}
\label{even claim}

$\displaystyle\order_{k=\frac{n}{2}}^1 Block(k)$ is a shelling for $T(\textnormal{M}_n^{\otimes})$. \\ 

Where $Block(k) := \displaystyle\order_{I \in \mathcal{K}^{(k)}}O(D_{I}^n)$. \\

\begin{PC}

Let $T,S \in T(\textnormal{M}_n^{\otimes})$ and suppose $S$ precedes $T$ in the ordering. Then $T \in O(D_{I_1}^n)$ and $S \in O(D_{I_2}^n)$ for some $I_1,I_2  \in \mathcal{P}(\mathcal{K})$ where $|I_1| \leq |I_2|$.\\

If there is a region $R$ in $T$ that contains a special mutable arc $\gamma$, such that $\gamma$ is not an arc in $S$, then $\mu_{\gamma}(T)$ precedes $T$ in the ordering and $S\cap T \subseteq \mu_{\gamma}(T)\cap T$.

So suppose that for \underline{every} region $R$ of $T$ \underline{all} special mutable arcs in that region are also arcs in $S$.  Then by Lemma \ref{fix even} $I_2 \subseteq I_1$. Since $|I_1| \leq |I_2|$ we must have $I_1 = I_2$.

If $O(D_{I}^n)$ was a shelling for $D_{I}^n$ then the proof would be finished. However, in general, it is not. To understand how we should proceed let us consider $D_{\{(1,\frac{n}{2}+1)\}}^n$. \\ 

By definition, $O(D_{\{(1,\frac{n}{2}+1)\}}^n) = \displaystyle\order_{i=1}^2 S(\mathcal{X}_i^n)$. Let $T$ be the first triangulation of $S(\mathcal{X}_2)$ and let $S \in S(\mathcal{X}_1)$. Corollary \ref{intersection} tells us that the only arc $T$ and $S$ share in common is the diagonal c-arc $(1,\frac{n}{2}+1)$. If $n=2$ then $O(D_{\{(1,2)\}}^2) = S,T$ is a shelling for $D_{\{(1,2)\}}^n$. However, if $n \geq 4$ then there are at least 4 arcs in $S$ and $T$. Hence, $\mu_{\gamma}(T) \notin S(\mathcal{X}_1^n)$ for any arc $\gamma$ in $T$, since $\mu_{\gamma}(T)$ and $S$ can share at most two arcs in common. \\

However, as $n \geq 4$ the first triangulation of $S(\mathcal{X}_2^n)$ contains (at least one) arc $\gamma$ that mutates to a diagonal c-arc. And so $\mu_{\gamma}(T)$ contains more diagonal c-arcs than $T$. Hence $\mu_{\gamma}(T)$ precedes $T$ in the overall ordering for $T(\textnormal{M}_n^{\otimes})$.

\end{PC}
\hfill \textit{End of proof of Claim \ref{even claim}.}

\end{claim}

\end{proof}

\subsection{Shellability of $T(\textnormal{M}_n^{\otimes})$ for odd $n$.}

In the even case diagonal arcs were a key ingredient in the shelling of $T(\textnormal{M}_n^{\otimes})$. We will see 'diagonal triangles' play the same role in the odd case. For the duration of this section we fix $n = 2k+1$.

\begin{defn}

A triangle in $\textnormal{M}_n$ comprising of two c-arcs $(i, i+k)$, $(i,i+k+1)$ and the boundary segment $(i+k,i+k+1)$ for some $i \in [1,n]$ is called a \textit{\textbf{diagonal triangle}} (d-triangle). Additionally, call $i$ the \textit{\textbf{special vertex}} of the d-triangle.

\end{defn}

\begin{defn}

Let $\mathcal{Y}^n$ be the partial triangulation of $\textnormal{M}_n$ containing the d-triangle $(k+1,2k+1,1)$. And let $T(\mathcal{Y}^n) \subseteq T(\textnormal{M}_n^{\otimes})$ consist of all c-triangulations of $\textnormal{M}_n$ containing the d-triangle $(k+1,2k+1,1)$, and no other d-triangles. See Figure \ref{fig:typey}.

\begin{figure}[H]
\centering
\includegraphics[width=40mm]{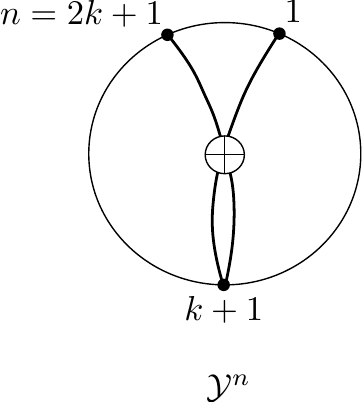}
\caption{}
\label{fig:typey}
\end{figure} 

\end{defn}

\begin{defn}

Let $T \in T(\mathcal{Y}^n)$ and $\gamma$ a c-arc in $T$. $\gamma = (i,j)$ for some $i \in [1, k+1]$ and $j \in [k+1,n]$. Define the \textit{\textbf{length}} of $\gamma$ as $l(\gamma) := j-i+1$, see Figure \ref{fig:oddlength}.

\begin{figure}[H]
\centering
\includegraphics[width=45mm]{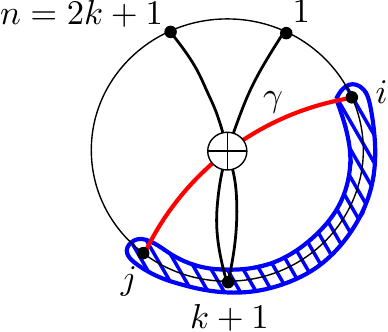}
\caption{'Number of marked points in shaded tube' $= l(\gamma)$.}
\label{fig:oddlength}
\end{figure}

\end{defn}

\begin{lem}
\label{max odd}
The max length of any c-arc in $T \in  \mathcal{Y}^n$ is $k+1$.

\begin{proof}

Given $T \in  T(\mathcal{Y}^n)$ we will prove by induction on $i \in [1,k+1]$ that there is no c-arc in $T$, with endpoint $k+i$, of length greater than $k+1$. For $i=1$ this trivially holds. Now assume the statement is true for $i$. Then there is a c-arc $\gamma = (x,k+i)$ in $T$ where $x \in [i,k+1]$. But the c-arc of maximum length, with endpoint $k+i+1$, that is compatible with $\gamma$ is $\beta = (x,k+i+1)$. If $x \in [i+1,k+1]$ then indeed $l(\beta) \leq k+1$. If $x=i$ then we have a d-triangle $(i,k+i,k+i+1)$ with special vertex $i$ - which is forbidden. So indeed $l(\beta) \leq k+1$.
\end{proof}

\end{lem}

\begin{lem}

$T(\mathcal{Y}^n) \equiv T(\mathcal{X}_1^{n+1})$. As such, $T(\mathcal{X}_1^{n+1})$ induces an upper shelling of $T(\mathcal{Y}^n)$. Denote this upper shelling by $S(\mathcal{Y}^n)$.

\begin{proof}

Add a marked point to the d-triangle $(k+1,2k+1,1)$ in $\mathcal{Y}^n$ and relabel the marked points. Adding the c-arc $(1,k+2)$ we get $\mathcal{X}_1^{n+1}$. Lemma \ref{max odd} tells us the maximum length of an arc in $T \in T(\mathcal{Y}^n)$ is $k+1$. And since the length of a max arc in $T(\mathcal{X}_1^{n+1})$ is also $k+1$ then $T(\mathcal{Y}^n) \equiv T(\mathcal{X}_1^{n+1})$. See Figure \ref{fig:oddshelling}.

\end{proof}

\begin{figure}[H]
\centering
\includegraphics[width=130mm]{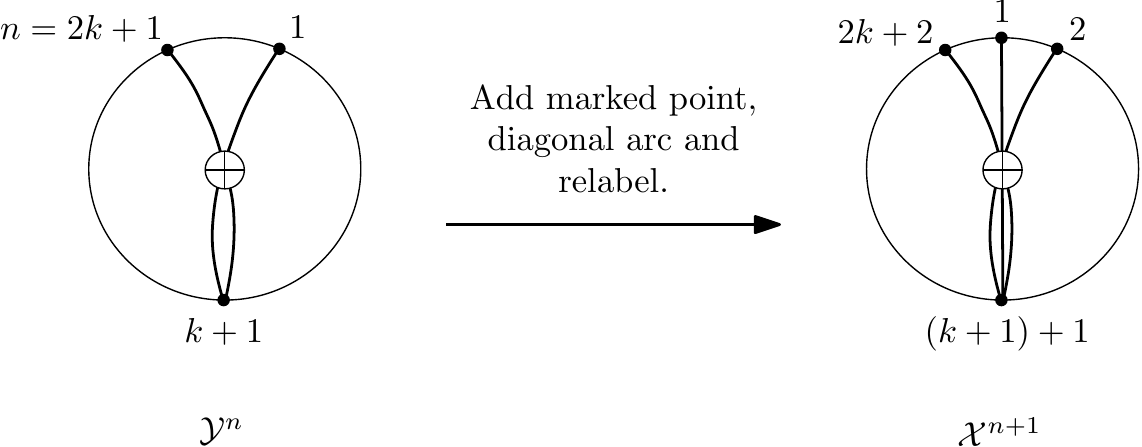}
\caption{}
\label{fig:oddshelling}
\end{figure} 

\end{lem}

\begin{lem}
\label{odd region}
For any $T \in  T(\mathcal{Y}^n)$ there are an odd number of d-triangles in $T$. Moreover, the collection of triangulations of any region cut out inbetween d-triangles, such that no other d-triangles occur, is equivalent to $T(\mathcal{Y}^m)$ for some $m < n$.

\begin{proof}
We will show that if there are two d-triangles there must in fact be a third. Additionally we'll show the collection of (legitimate) triangulations in any region cut out inbetween the three d-triangles is equivalent to $T(\mathcal{Y}^m)$ for some $m < n$. And applying induction on this we will have proved the lemma. \\

Suppose there are at least two d-triangles in a c-triangulation $T$. Without loss of generality we may assume the two d-triangles $(k+1,2k+1,1)$ and $(i,i+k,i+k+1)$ are in $T$, for some $i \in [1,k]$. See Figure \ref{fig:2triangles}. 

\begin{figure}[H]
\centering
\includegraphics[width=50mm]{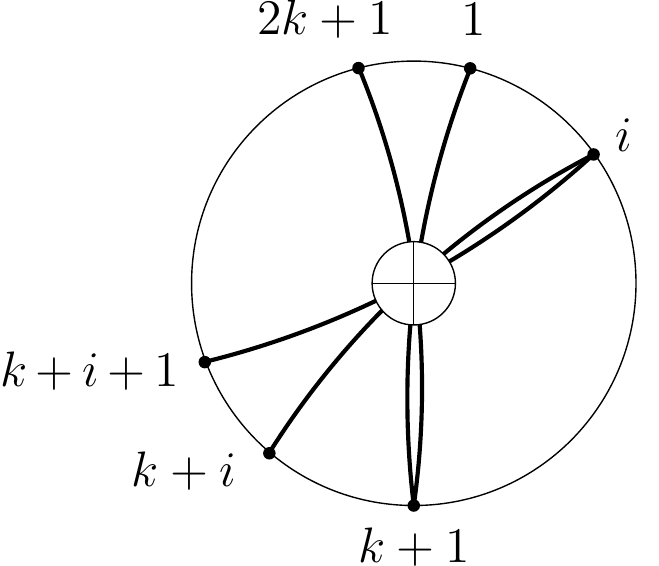}
\caption{}
\label{fig:2triangles}
\end{figure} 

We will show there is a third d-triangle with special vertex $z \in [i+k+1,2k+1]$. Note that if $(i+1,i+k+1) \in T$ then the d-triangle $(i+k+1,i,i+1) \in T$. Similarly, if $(k,2k+1) \in T$ then the d-triangle $(2k+1,k+1,k) \in T$.

So suppose $(i+1,i+k+1),(k,2k+1) \notin T$. This then implies $(i+1,x) \in T$ for some $x \in [i+k+2,2k]$, and $(k,y) \in T$ for some $y \in [i+k+2,2k]$. In turn, by induction, there is a d-triangle with special vertex $z \in [x,y]$. See Figure \ref{fig:exist3}.

\begin{figure}[H]
\centering
\includegraphics[width=60mm]{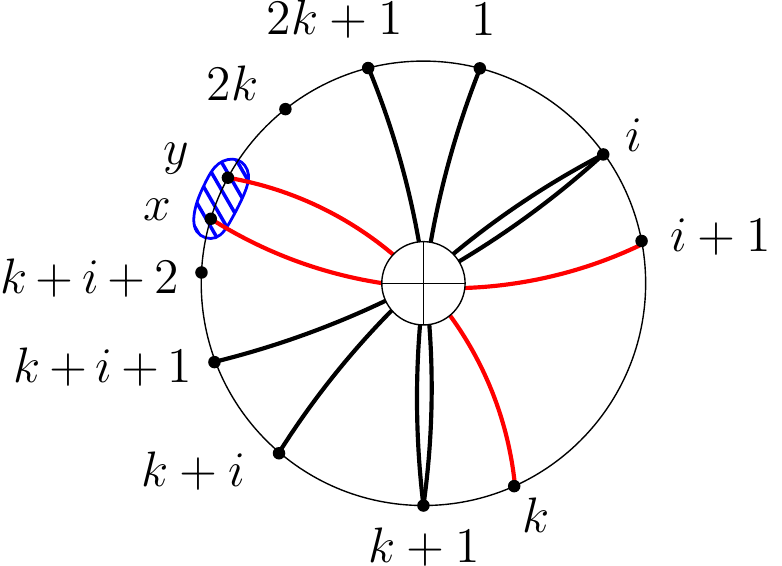}
\caption{By induction there is a d-triangle with its special vertex in the shaded region.}
\label{fig:exist3}
\end{figure}

What remains to prove is that each region cut out by these three d-triangles is equivalent to $T(\mathcal{Y}^m)$ for some $m < n$.

Consider the d-triangles $(k+1,2k+1,1)$ and $(i,i+k,i+k+1)$ with special vertices $k+1$ and $i$, respectively. Let $R$ be the region bounded by the c-arcs $(1,k+1)$, $(i,i+k)$ and the boundary segments $[1,i]$,$[k+1,k+i]$. Collapsing the boundary segment $[i,k+1]$ to a point and collapsing $[k+i,1]$ to a boundary segment preserves the notion of length in $R$. After collapsing we see that triangulating $R$ (so that no d-triangles occur) is equivalent to triangulating $\mathcal{Y}^{2i-1}$. See Figure \ref{fig:collapsedtriangle}.

\begin{figure}[H]
\centering
\includegraphics[width=135mm]{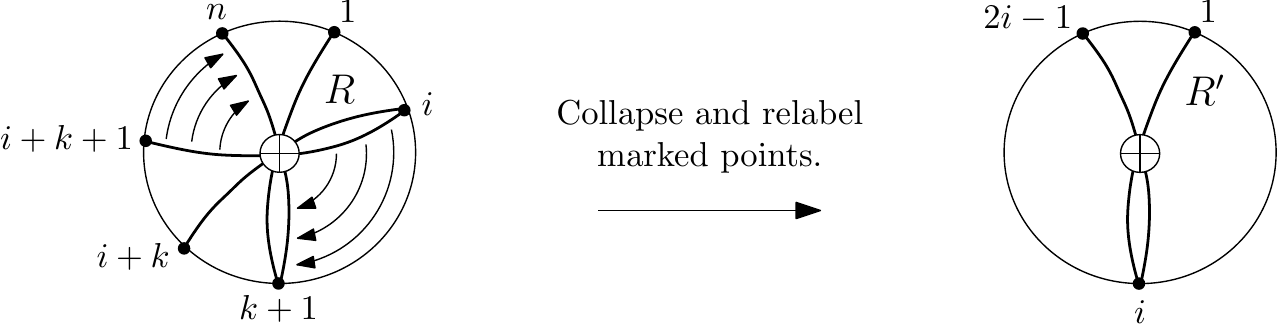}
\caption{}
\label{fig:collapsedtriangle}
\end{figure} 

Similarly the collection of triangulations of either of the other two regions cut out by the three d-triangles is equivalent to $T(\mathcal{Y}^m)$ for some $m < n$. This completes the proof.
\end{proof}

\end{lem}

\begin{defn}

Let $T \in T(\mathcal{Y}^n)$ and let $\gamma$ be a c-arc in $T$. Call $\gamma$ \textit{\textbf{special mutable}} if it is upper mutable or $\mu_{\gamma}(T)$ contains more d-triangles than $T$.

\end{defn}

\begin{lem}
\label{fix odd}
Let $T \in T(\mathcal{Y}^n)$ and let $P_T$ be the partial triangulation of $\textnormal{M}_n$ consisting of all special mutable arcs in $T$. Then for any triangulation of $P_T$ there is no d-triangle with special vertex $i$ $\forall i \in [1,\dots,n]\setminus k+1$.

\begin{proof}

We follow the same idea used in Lemma \ref{fix even}. Namely, we will prove the lemma by induction on the shelling order of $S(\mathcal{Y}_n)$.

Let $T_1$ be the first triangulation in the shelling. Note $\gamma_i = (i,k+i-1)$ is a special mutable c-arc in $T_1$ $\forall i \in [2,k+1]$. Moreover $\gamma_i$ is not compatible with the c-arc $(i-1, k+i)$. Hence there is no d-triangle with special vertex $i-1$ or $k+i$ $\forall i \in [2,k+1]$. This proves the base inductive case.\\

Let $T \in T(\mathcal{Y}^n)$. What remains to show is that for any lower mutable arc $\gamma \in T$, the d-triangles incompatible with $P_T$ are precisely the d-triangles incompatible with $P_{\mu_{\gamma}(T)}$. 

So let $\gamma$ be a lower mutable arc in $T$. Let $\beta_1, \beta_2$ be the c-arcs of the quadrilateral containing $\gamma$. See Figure \ref{fig:fixspecialodd}. 

\begin{figure}[H]
\centering
\includegraphics[width=35mm]{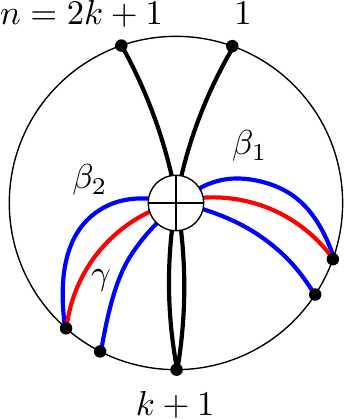}
\caption{}
\label{fig:fixspecialodd}
\end{figure}

Note that $\beta_1$ and $\beta_2$ could be upper mutable in $T$, but they will definitely not be upper mutable in $\mu_{\gamma}(T)$. Analogous to the proof of Lemma \ref{fix even}, to prove the lemma it suffices to show $\mu_{\gamma}$ is incompatible with all the d-triangles incompatible with either $\beta_1$ or $\beta_2$.

This follows from the fact that a c-arc $\alpha = (x,k+y)$ of length less than $k$ is incompatible with d-triangles with special vertex $z \in [y,x-1]\cup[k+y+1,k+x]$. See Figure \ref{fig:oddstop}.

\begin{figure}[H]
\centering
\includegraphics[width=55mm]{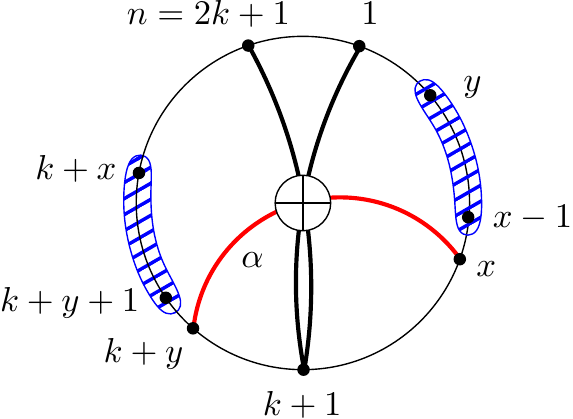}
\caption{$\alpha$ is incompatible with d-triangles whose special vertex lies in one of the shaded regions.}
\label{fig:oddstop}
\end{figure}

\end{proof}

\end{lem}

An analogous argument to Lemma \ref{at least one diagonal} proves the following lemma.

\begin{lem}

In each c-triangulation $T$ of $\textnormal{M}_n$ there is at least one d-triangle.

\end{lem}

\begin{lem}
\label{odd}
$T(\textnormal{M}_n^{\otimes})$ is shellable for odd $n$.

\begin{proof}

Let $\mathcal{K}$ be the collection of d-triangles of $\textnormal{M}_n$ that can occur in a triangulation without containing any other d-triangles. Consider $I = \{ \Delta_1, \ldots, \Delta_k \} \subseteq \mathcal{K}$ and let $D_{I}^n$ consist of all triangulations of $T(\textnormal{M}_n^{\otimes})$ containing every d-triangle in $I$, and no d-triangles in $\mathcal{K}\setminus I$.

By Lemma \ref{odd region}, each region cut out inbetween the d-triangles in $I$ is shellable. Taking the product of these shellings over all regions gives us a shelling for $D_{I}^n$. Denote this shelling by $S(D_{I}^n)$. \\

\begin{claim}
\label{odd claim}

$\displaystyle\order_{k=\frac{n}{2}}^1 Block(k)$ is a shelling for $T(\textnormal{M}_n^{\otimes})$. \\ 

Where $Block(k) := \displaystyle\order_{I \in \mathcal{K}^{(k)}}S(D_{I}^n)$. \\

\begin{PC}

Let $T,S \in T(\textnormal{M}_n^{\otimes})$ and suppose $S$ precedes $T$ in the ordering. Then $T \in S(D_{I_1}^n)$ and $S \in S(D_{I_2}^n)$ for some $I_1,I_2  \in \mathcal{P}(\mathcal{K})$ where $|I_1| \leq |I_2|$.

If there is a region $R$ in $T$ that contains a special arc $\gamma$, such that $\gamma$ is not an arc in $S$, then $\mu_{\gamma}(T)$ precedes $T$ in the ordering and $S\cap T \subseteq \mu_{\gamma}(T)\cap T$ .

So suppose that for \underline{every} region $R$ of $T$ \underline{all} special arcs in that region are also arcs in $S$.  Then by Lemma \ref{fix odd} $I_2 \subseteq I_1$. Since $|I_1| \leq |I_2|$ we must have $I_1 = I_2$. And since $S(D_{I_1}^n)$ is a shelling for $D_{I}^n$ the claim is proved.

\end{PC}
\hfill \textit{End of proof of Claim \ref{odd claim}.}

\end{claim}

\end{proof}

\end{lem}

Together Lemma \ref{even} and Lemma \ref{odd} prove $T(\textnormal{M}_n^{\otimes})$ is shellable for all $n \geq 1$.

Returning to our example of $\text{M}_3$, Figure \ref{fig:numberedM3} shows a shelling of $T(\text{M}^{\circ}_3)$ that we can obtain through our construction.

\begin{figure}[H]
\centering
\includegraphics[width=100mm]{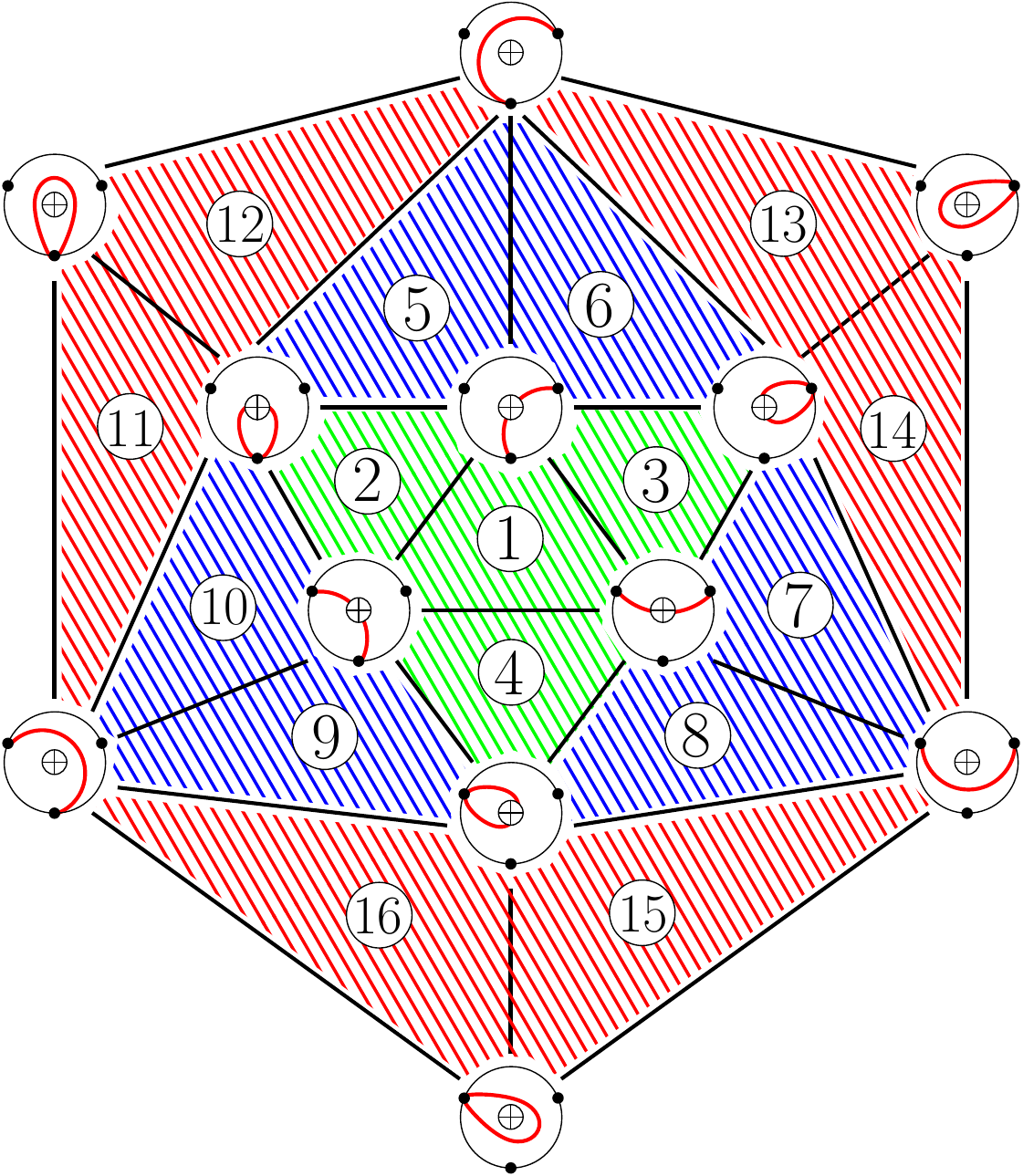}
\caption{Shelling of $T(\text{M}^{\circ}_3)$}
\label{fig:numberedM3}
\end{figure}

\subsection{Proof of Main Theorem.}

\begin{thm}[Main Theorem]
\label{Main}
$Arc(\textnormal{M}_n)$ is shellable for $n \geq 1$.

\begin{proof}

Let $\mathcal{C}$ consist of all quasi-triangulations of $\textnormal{M}_n$ containing the one-sided closed curve. Cutting along the one-sided curve in $\textnormal{M}_n$ we are left with the marked surface $C_{n,0}$. $Arc(C_{n,0})$ is shellable by Proposition \ref{shellingcn}. Since $\mathcal{C}$ is the cone over $Arc(C_{n,0})$, then by Proposition \ref{join} it is also shellable. Let $S(\mathcal{C})$ denote a shelling for $\mathcal{C}$. Lemma \ref{half} together with Lemma \ref{even} and Lemma \ref{odd} proves that $T(\text{M}^{\circ}_n)$ is shellable. Let $S(\textnormal{M}_n^{\circ})$ be a shelling of $T(\text{M}^{\circ}_n)$.

\begin{claim}
\label{main claim}

$S(\textnormal{M}_n):= S(\textnormal{M}_n^{\circ}), S(\mathcal{C})$ is a shelling for $T(\textnormal{M}_n)$

\begin{PC}

Suppose $S,T \in S(\textnormal{M}_n)$ and $S$ precedes $T$ in the ordering. Without loss of generality we may assume $S \in S(\textnormal{M}_n^{\circ})$ and $T \in S(\mathcal{C})$. Since $T$ contains the one-sided closed curve $\gamma$, and $\gamma \notin S$ then $S\cap T \subseteq \mu_{\gamma}(T)\cap T$. Moreover, $ \mu_{\gamma}(T) \in S(\textnormal{M}_n^{\circ})$ so precedes $T$ in the ordering.

\end{PC}
\hfill \textit{End of proof of Claim \ref{main claim}.}
\end{claim}

\end{proof}

\end{thm}

\begin{cor}
\label{sphere}

$Arc(\textnormal{M}_n)$ is a PL $(n-1)$-sphere for $n \geq 1$.

\begin{proof}

Follows immediately from Theorem \ref{Ziegler} and Theorem \ref{Main}.

\end{proof}

\end{cor}

\Addresses

\end{document}